\documentclass[a4paper, 11pt]{article}
\usepackage{fancyhdr}
\usepackage{graphicx}
\usepackage{geometry}

\usepackage[T1]{fontenc}
\usepackage[latin1]{inputenc}

\usepackage{color}

\usepackage{amsmath}
\usepackage{amssymb}
\usepackage{amsfonts}
\usepackage{amsthm,amscd}

\newtheorem{theorem}{Theorem}[section]
\newtheorem{lemma}[theorem]{Lemma}

\newtheorem{proposition}[theorem]{Proposition}
\theoremstyle{definition}

\newtheorem{remark}{Remark}[section]

\numberwithin{equation}{section}
\def\N{\mbox{I\hspace{-.15em}N}}
\def\R{\mbox{I\hspace{-.15em}R}}

\begin{document}

\date{}
\title{A class of stochastic differential equations with super-linear growth and non-Lipschitz coefficients\thanks{Partially
supported by PHC Volubilis MA/10/224, PHC Tassili 13MDU887 and  MODTERCOM project  APEX Programme r\'egion Provence-Alpe-C\^ote d'Azur.}}
\author{\small{Khaled BAHLALI}\thanks{\textbf{bahlali@univ-tln.fr}, Universit\'{e} de Toulon, IMATH, EA 2134, 83957
La Garde, France, \& CNRS, I2M, Universit\'{e} Aix Marseille (2013/2014), Marseille, France.}\, ,
\small{Antoine HAKASSOU}\thanks{\textbf{antoinehakassou@gmail.com}, Universit\'{e} Priv\'{e}e de Marrakech,
Ecole d'Ing\'{e}nierie et d'Innovation, Km 13, route d'Amezmiz, BP 42312, Marrakech, Maroc.}\, \
 and \small{Youssef OUKNINE}\thanks{\textbf{ouknine@uca.ma}, Universit\'{e} Cadi Ayyad de Marrakech \& Acad\'{e}mie Hassan II des
Sciences et Techniques de Rabat, Maroc.}
}

\maketitle

\begin{center}
\textbf{Abstract}
\end{center}

The purpose of this paper is to study some properties of solutions to one dimensional as well as multidimensional 
stochastic differential equations (SDEs in short) with super-linear growth conditions on the coefficients.
Taking inspiration from \cite{BEHP, KBahlali, Bahlali}, we introduce a new {\it{local condition}}  which ensures the pathwise 
uniqueness, as well as the non-contact property. We moreover show that the solution produces a stochastic flow of continuous 
maps and satisfies a large deviations principle of Freidlin-Wentzell type.
Our conditions on the coefficients go beyond the existing ones in the literature.
For instance, the coefficients are not assumed uniformly continuous and therefore can not satisfy
the classical Osgood condition. The drift coefficient could not be
locally monotone and the diffusion is neither locally Lipschitz nor uniformly elliptic.
Our conditions on the coefficients are, in some sense, near the best possible. 
Our results are sharp and mainly based on Gronwall lemma and the localization of the time parameter in
concatenated intervals.\\

{\bf{Keywords}}: Stochastic flows, Large deviations,
Non-Lipschitz coefficients, Pathwise uniqueness, Non-confluence, Euler scheme, Gronwall lemma.\\

{\bf AMS Subject Classification }: 60H10, 60J60, 34A12, 34A40.

\section{Introduction and motivations}

This work was initially motivated by the study of stochastic flows of homeomorphisms and large deviations of the following
simple example of one dimensional stochastic differential equations with super-linear growth coefficients:
\begin{equation}\label{xlogxintro}
X_{t}=x+\int_{0}^{t} X_{s} \log{|X_{s}|} ds +\int_{0}^{t} X_{s}\sqrt{|\log{|X_{s}|}|} dW_{s}
\end{equation}
where $x\in\mathbb{R}$ and $(W_{t})_{t \geq 0}$ is an $\mathbb{R}$-valued standard Brownian motion.

Our motivation for SDE (\ref{xlogxintro}) comes from the fact that the stochastic flows of homeomorphisms
defined by these type of SDEs may be related to the construction of  Canonical diffusions above the diffeomorphism group of the circle and also the construction of a metric in the H\"older-Sobolev space
$\mathcal{H}^{\frac{3}{2}}$, see Malliavin \cite{malliavin}.
Note also that the logarithmic nonlinearities $x\sqrt{\log{|x|}}$ and $x\log{|x|}$ are interesting in their own
since they are neither locally monotone nor uniformly continuous. They are, in some sense, near the best possible. Indeed,

1) An exponential transformation formally shows that the SDE with diffusion coefficient   $x\sqrt{\log{|x|}}$ is equivalent to 
the SDE with diffusion coefficient $|x|^{\frac12}$ which is the best possible for pathwise uniqueness according to Yamada \& Watanabe \cite{yawa}.

2) The growth conditions $x\log{|x|}$ on the drift coefficient constitute a critical case
in the sense that, for any $\varepsilon >0$, the solutions of the ordinary
differential equation $ X_t = x + \int_0^t X_s^{1+\varepsilon}ds$ \ explode
at a finite time.

Note finally that the nonlinearity $u\log{|u|}$ also appear  in some PDEs arising in physics, see e.g.
 \cite{Bialynicki, Mycielski, Cazenave, PZ}.

We now begin with our subject. Let $\sigma:\mathbb{R}^{d}\rightarrow\mathbb{R}^{d}\times\mathbb{R}^{m}$ and 
$b:\mathbb{R}^{d}\rightarrow\mathbb{R}^{d}$ be
respectively matrix-valued and vector-valued continuous functions and consider the following forward It\^o SDE:
\begin{equation}\label{1}
X_{t}=x+\int_{0}^{t} b(X_{s})ds +\int_{0}^{t} \sigma(X_{s}) dW_{s}
\end{equation}
where $x\in\mathbb{R}^{d}$ is fixed and $(W_{t})_{t\geq0}$ is an $\mathbb{R}^{m}$-valued standard Brownian motion defined
on a complete filtered probability space $(\Omega, \mathcal{F}, (\mathcal{F}_{t}), \mathbb{P})$ with $(\mathcal{F}_{t})$ 
a right continuous increasing family of sub-$\sigma$-fields of $\mathcal{F}$ each containing $\mathbb{P}$-null sets.

According to Skorohod result \cite{Skorohod}, SDE \eqref{1} admits a weak (in law) solution up to an explosion time
 (see also Ikeda \& Watanabe \cite{Ikeda}, Karatzas \& Shreve \cite{Karatzas},
Revuz \& Yor \cite{Revuzyor}, Stroock \& Varadhan \cite{varadhan}). Thanks to the celebrated result of Yamada
 \& Watanabe \cite{yawa}, we know that if a weak solution is pathwise unique, it is then  a strong solution, that is adapted to the Brownian filtration.
Having a unique strong solution, it becomes possible to study some other properties such as the dependence to the initial data
and the large deviations of Freidlin-Wentzell's type. So the study of  pathwise uniqueness is greatly interesting.

Under Lipschitz conditions, it is classical that the pathwise
uniqueness holds, see for instance It\^o \cite{Ito}, Yamada \& Watanabe \cite{yawa}, and the non-contact
property (also known as non-confluence property) of the solutions holds, see Emery \cite{Emery},
Kunita \cite{Kunita}, Meyer \cite{Meyer}, Yamada \& Ogura \cite{yaogura}. Moreover, the solution depends
bicontinuously on $(t,x)$, see Kunita \cite{Kunita}, and satisfies
a large deviations principle of Freidlin-Wentzell type, see Freidlin \& Wentzell \cite{FW}, Azencott \cite{Azencott},
Dembo \& Zeitouni \cite{Dembozeitouni}, Deuschel \& Stroock \cite{Deuschel-Stroock}.

In the last 15 years, the study of SDEs with few regularities on the coefficients
has a renewed interest, see for instance  \cite{Athreya, bahlali1999,  Bass-Perk, Fang, Zhang, KR, liang, liang2, 
RenZhang, RenZhang2, swart1, swart2, XichengZhang}.

The purpose of this paper is to study multidimensional SDEs with logarithmic nonlinearity growth and our guiding example is the
one-dimensional SDE (\ref{xlogxintro}).

  Let \  $|\cdot|$ denote the Euclidean distance in $\mathbb{R}^{d}$,
$||\sigma||^{2}=\sum_{i=1}^{d}\sum_{j=1}^{m}\sigma_{ij}^{2}$, and for any integer $N>e$ \ we put
$\mathbb{B}(N):=\{x\in\mathbb{R}^{d}; |x|\leq N\}$. 
We now introduce our main assumption which is inspired from the papers \cite{BEHP, KBahlali, Bahlali} and which cover our 
motivating SDE (\ref{xlogxintro}).

\begin{equation}\tag{H1} \label{H1}
\left\{
\begin{array}{cc}
\mbox{there exist} \ C>0 \ \mbox{and} \ \mu>0 \  \mbox{such that for every} \ x,y\in \mathbb{B}(N), \\ \\
 ||\sigma(x)-\sigma(y)||\leq C\sqrt{\log{N}} |x-y| + C(\log{N})/N^{\mu}\\ \\
  |b(x)-b(y)| \leq C\log{N} |x-y| + C(\log{N})/N^{\mu}
\end{array}
 \right.
\end{equation}

 We first establish that assumption (H1) ensures the existence of a pathwise unique solution for SDE \eqref{1}.
Then we prove that, under this assumption, the  solution  has the non-contact property. Moreover, this solution depends continuously 
in its two variables $(t,x)$ and satisfies a large deviations principle of
Freidlin-Wentzell type.
In some sense, assumption (H1) is near the best possible. Moreover, our methods of proving the pathwise
uniqueness, the non-contact property, the bicontinuity and the large deviations are simples. Also, they work in any finite dimension 
and improve those of \cite{Fang, Zhang, liang, liang2}.

The rest of the paper is organized as follows. In section 2, we prove the pathwise uniqueness,
a one-dimensional comparison theorem, the non-contact property and the bicontinuity of the solution of SDE \eqref{1}.
In section 3, dealing with Euler scheme, we establish that the solution satisfies a large deviations
principle of Freidlin-Wentzell type.
Finally, as a by-product of our results,  we study in section 4 our guide-motivating SDE (\ref{xlogxintro}). 
In the end of section 4, we show that our guidance SDE (\ref{xlogxintro}) is not covered by the paper \cite{Fang, Zhang}.  
We also show that our paper cover the papers \cite{Fang, liang, liang2}.

\begin{remark}

Throughout the paper, the universal constants appearing in the inequalities are denoted by C and allowed
to change from place to place.
Moreover, all processes considered in the sequel, if it is not precised, will be assumed to be defined
on the complete filtered probability space
$(\Omega, \mathcal{F}, (\mathcal{F}_{t}), \mathbb{P})$.

\end{remark}

\section{Stochastic flows of continuous maps}

The main purpose of this section is to prove that under hypothesis \eqref{H1}, the SDE \eqref{1} has a unique strong solution
which produces a stochastic flows of continuous maps. In this goal, we shall establish the pathwise uniqueness,
the non-contact property and the bicontinuous dependence of the solution to the initial values.

\subsection{Pathwise uniqueness}

We give as follows the capital result of this section.

\begin{theorem}\label{thm 2.1}

Assume that hypothesis \eqref{H1} holds and let $(X_{t}(\omega))$ and $(Y_{t}(\omega))$ be two solutions (of continuous samples) 
without explosion of the SDE \eqref{1} such that $X_{0}(\omega)=Y_{0}(\omega)=x$ almost surely.
Then, for any $T>0$ we have almost surely $X_{t}(\omega)=Y_{t}(\omega)$ for all $0\leq t\leq T$.

\end{theorem}

\begin{proof}{}
Let $(X_{t}(\omega))$ and $(Y_{t}(\omega))$ be two solutions (of continuous samples) of the SDE (\ref{1}) with the same 
initial datum $x\in\mathbb{R}^{d}$.

For  $N\in\mathbb{N}^{\ast}$, we define the stopping time $\zeta_{N}:=\inf\{t>0; |X_{t}|>N \mbox{ or } |Y_{t}|>N \}$.
Since the solutions of SDE \eqref{1} are assumed to be conservative, then $\zeta_{N}$ tends
to $+\infty$ as N tends to $+\infty$.

Using It\^o's formula, we  get
\begin{equation*}
\begin{split}
|X_{t \wedge \zeta_{N}}-Y_{t \wedge \zeta_{N}}|^{2}=2 \int_{0}^{t\wedge \zeta_{N}}
\langle (X_{s \wedge \zeta_{N}}-Y_{s \wedge \zeta_{N}}),
(\sigma (X_{s \wedge \zeta_{N}})-\sigma (Y_{s \wedge \zeta_{N}}))dW_{s} \rangle\\
+ 2\int_{0}^{t\wedge \zeta_{N}}
\langle (X_{s \wedge \zeta_{N}}-Y_{s \wedge \zeta_{N}}), (b(X_{s \wedge \zeta_{N}})-b(Y_{s \wedge \zeta_{N}})) \rangle ds \\
+ \int_{0}^{t\wedge \zeta_{N}} {||\sigma (X_{s \wedge \zeta_{N}})-\sigma (Y_{s \wedge \zeta_{N}})||}^{2}ds.
\end{split}
\end{equation*}

Thanks to the Burkholder inequality, we get for any $T>0$
\begin{equation*}
\begin{split}
\mathbb{E} \sup_{t \leq T}{| X_{t \wedge \zeta_{N}}-Y_{t \wedge \zeta_{N}} |}^{2} \leq
2 \mathbb{E}\int_{0}^{T} |b(X_{s \wedge \zeta_{N}})-b(Y_{s \wedge \zeta_{N}})||X_{s \wedge \zeta_{N}}-Y_{s \wedge \zeta_{N}}| ds \\
+2 C_{1} \mathbb{E} {\left( \int_{0}^{T} {||\sigma (X_{s \wedge \zeta_{N}})-\sigma (Y_{s \wedge \zeta_{N}})||}^{2}
{| X_{s \wedge \zeta_{N}}-Y_{s \wedge \zeta_{N}} |}^{2}ds \right)}^{1/2} \\
+\mathbb{E}\int_{0}^{T} {||\sigma (X_{s \wedge \zeta_{N}})-\sigma (Y_{s \wedge \zeta_{N}})||}^{2}ds.
\end{split}
\end{equation*}

This implies that
\begin{equation*}
\begin{split}
\mathbb{E} \sup_{t \leq T}{| X_{t \wedge \zeta_{N}}-Y_{t \wedge \zeta_{N}} |}^{2} \leq
2 \mathbb{E}\int_{0}^{T} |b(X_{s \wedge \zeta_{N}})-b(Y_{s \wedge \zeta_{N}})||X_{s \wedge \zeta_{N}}-Y_{s \wedge \zeta_{N}}| ds \\
+ (1+2C_{1}^{2})\mathbb{E}\int_{0}^{T} {||\sigma (X_{s \wedge \zeta_{N}})-\sigma (Y_{s \wedge \zeta_{N}})||}^{2}ds \\
+\frac{1}{2}\mathbb{E} \sup_{t \leq T}{| X_{t \wedge \zeta_{N}}-Y_{t \wedge \zeta_{N}} |}^{2}.
\end{split}
\end{equation*}

Then
\begin{equation*}
\begin{split}
\mathbb{E}\sup_{t \leq T}{| X_{t \wedge \zeta_{N}}-Y_{t \wedge \zeta_{N}} |}^{2} \leq
(2+4C_{1}^{2})\mathbb{E}\int_{0}^{T} {||\sigma (X_{s \wedge \zeta_{N}})-\sigma (Y_{s \wedge \zeta_{N}})||}^{2}ds \\
+4\mathbb{E}\int_{0}^{T} |b(X_{s \wedge \zeta_{N}})-b(Y_{s \wedge \zeta_{N}})| |X_{s \wedge \zeta_{N}}-Y_{s \wedge \zeta_{N}}| ds.
\end{split}
\end{equation*}

According to hypothesis \eqref{H1}, it follows that
 \begin{equation*}
\mathbb{E} \sup_{t \leq T}{| X_{t \wedge \zeta_{N}}-Y_{t \wedge \zeta_{N}} |}^{2} \leq  CT\frac{\log{N}}{N^{\mu}}
+ C\log{N}\int_{0}^{T} \mathbb{E} \sup_{u\leq s} {| X_{u \wedge \zeta_{N}}-Y_{u \wedge \zeta_{N}} |}^{2}ds.
\end{equation*}

By the Gronwall lemma, we get
\begin{equation}\label{4}
\mathbb{E} \sup_{t \leq T}{|X_{t \wedge \zeta_{N}}-Y_{t \wedge \zeta_{N}}|}^{2} \leq CT\frac{\log{N}}{N^{\mu-CT}}.
\end{equation}

Since
 \begin{equation*}
  \sup_{t\leq T} |X_{t}-Y_{t}|^{2}1_{\{T\leq \zeta_{N}\}}
  =\sup_{t\leq T} |X_{t\wedge\zeta_{N}}-Y_{t\wedge\zeta_{N}}|^{2}1_{\{T\leq \zeta_{N}\}} \ \ a.s.
 \end{equation*}

Then,
 \begin{equation*}
  \sup_{t\leq T} |X_{t}-Y_{t}|^{2}1_{\{T\leq \zeta_{N}\}} \leq \sup_{t\leq T} |X_{t\wedge\zeta_{N}}-Y_{t\wedge\zeta_{N}}|^{2}\ \ a.s.
 \end{equation*}

Letting $N$ tends to $+\infty$ in the previous inequality and thanks to fact that $\zeta_{N}$ goes to $+\infty$ $a.s$,
it follows that
 \begin{equation*}
  \sup_{t\leq T} |X_{t}-Y_{t}|^{2} \leq
  \liminf_{N\rightarrow +\infty}\sup_{t\leq T} |X_{t\wedge\zeta_{N}}-Y_{t\wedge\zeta_{N}}|^{2}.
 \end{equation*}

Taking the expectation we get
 \begin{equation}
  \mathbb{E}\sup_{t\leq T}|X_{t}-Y_{t}|^{2} \leq
  \mathbb{E}\liminf_{N\rightarrow +\infty} \sup_{t\leq T}|X_{t\wedge\zeta_{N}}-Y_{t\wedge\zeta_{N}}|^{2}.
 \end{equation}

Using  Fatou's lemma and sending $N$ to $+\infty$ in \eqref{4}, it follows that for any $T<\mu/C$
 \begin{equation}
\mathbb{E} \sup_{t \leq T}{| X_{t}-Y_{t} |}^{2} =0.
\end{equation}

Starting again from $\mu/C$ and applying the same arguments as above, we get for any $T\in [\mu/C; 2\mu/C[$
\begin{equation*}
\mathbb{E} \sup_{t \leq T}{| X_{t}-Y_{t} |}^{2} =0.
\end{equation*}

For $k\in\mathbb{N}$, we set $T_{k}:=k\mu/C$. Clearly $T_{k}$ goes to $+\infty$ as k tends $+\infty$.
We start now from $T_{k}$ and then in a same manner as in the first part of the proof, we show that for any $T\in [T_{k}, T_{k+1}[$
\begin{equation*}
\mathbb{E} \sup_{t \leq T}{| X_{t}-Y_{t} |}^{2} =0.
\end{equation*}

Since for every $T\in\mathbb{R}_{+}$, there exists a unique $k_{0}\in\mathbb{N}$ such that $T\in [T_{k_{0}}, T_{k_{0}+1}[$ we get:
\begin{equation*}
\mathbb{E} \sup_{t \leq T}{| X_{t}-Y_{t} |}^{2}
\leq \sum_{k=0}^{k_{0}} \mathbb{E} \sup_{t\in [0,T]\cap [T_{k}, T_{k+1}[}{| X_{t}-Y_{t} |}^{2}=0.
\end{equation*}

Hence, for every $0\leq t\leq T$, we have $X_{t}=Y_{t}$ \ $a.s$. Thanks to the continuity of the samples paths,
the two solutions are indistinguishable.

\end{proof}

\begin{remark}

It should be noted that the conditions \eqref{H1} does not imply the non-explosion of the SDE \eqref{1}.
If the solution explodes at a finite time, Theorem \ref{thm 2.1}  ensures then the pathwise uniqueness up to a life-time.

\end{remark}

As a consequence of the pathwise uniqueness, we shall establish under additional conditions that the obtained unique strong solution
depends continuously to the initial data.

\begin{theorem}\label{thm 3.2}

Assume that the coefficients $\sigma$ and b are bounded and satisfy hypothesis \eqref{H1}.
Let $x_{l}\in\mathbb{R}^{d}$ be a sequence which converges to $x\in\mathbb{R}^{d}$ and consider $X_{t}(x_{l})$ and $X_{t}(x)$
the unique solutions of SDE \eqref{1} starting from $x_{l}$ and $x$ respectively.
Then, for any $T\geq0$, we have
\begin{equation*}
\lim_{l \rightarrow +\infty} \mathbb{E}\sup_{t \leq T} {|X_{t}(x_{l})-X_{t}(x)|}^{2}=0.
\end{equation*}

\end{theorem}

\begin{proof}{}
 Thanks to Theorem \ref{thm 2.1} the pathwise uniqueness holds. The proof follows then from
 \cite{BahlaliOuknineMezerdi}.
\end{proof}

\subsection{Comparison theorem}

Here, we prove a one-dimensional comparison theorem for the solutions of the SDE \eqref{1}.

\begin{theorem}\label{thmcompa}

Suppose, we are given the following:

$(i)$ a real continuous function $\sigma$ defined on $\mathbb{R}$ such that:
\begin{equation} \label{compa1}
 |\sigma(x)-\sigma(y)|\leq C\sqrt{\log{N}} |x-y| + C\frac{\log{N}}{N^{\mu}}
\end{equation}
for all $x,y\in B(N)=\{z\in\mathbb{R}^{d}; |z|\leq N\}$ for any integer $N>e$, and $C$, $\mu$ two positive reals,

$(2i)$ two real continuous functions $b_{1}$ and $b_{2}$ defined on $\mathbb{R}$ such that:
\begin{equation}\label{compa2}
 b_{1}(x)<b_{2}(x), \, \mbox{ for any } x\in\mathbb{R},
\end{equation}

$(3i)$ two real $\mathcal{F}_{t}$-adapted, continuous and conservative processes $x_{1}(t, \omega)$ and $x_{2}(t, \omega)$,

$(4i)$ a one-dimensional $\mathcal{F}_{t}$-Brownian motion $B(t, \omega)$ such that $B(0)=0$, a.s.,

$(5i)$ two real $\mathcal{F}_{t}$-adapted well measurable processes $\beta_{1}(t, \omega)$ and $\beta_{2}(t, \omega)$.\\
Assume that they satisfy the following condition with probability one:
\begin{equation}\label{compa3}
 x_{i}(t)-x_{i}(0)=\int_{0}^{t} \sigma(x_{i}(s))dB(s)+\int_{0}^{t} \beta_{i}(s)ds, \, i=1,2,
\end{equation}
\begin{equation}\label{compa4}
 x_{1}(0) \leq x_{2}(0)
\end{equation}
\begin{equation}\label{compa5}
 \beta_{1}(t) \leq b_{1}(x_{1}(t)) \mbox{ for all } t\geq0,
\end{equation}
\begin{equation}\label{compa6}
 \beta_{2}(t) \geq b_{2}(x_{2}(t)) \mbox{ for all } t\geq0.
\end{equation}

Then, with probability one, we have
\begin{equation}\label{compa}
 x_{1}(t) \leq x_{2}(t), \mbox{ for all } t\geq0.
\end{equation}

If furthermore, the pathwise uniqueness holds for at least one of the following stochastic differential equations:
\begin{equation}\label{compa7}
 dX_{t}=\sigma(X(t))dB(t)+b_{i}(X(t))dt, \, i=1,2,
\end{equation}
then, we have the same conclusion \eqref{compa} by weakening \eqref{compa2} to:
\begin{equation}\label{compa8}
 b_{1}(x) \leq b_{2}(x), \, \mbox{ for any } x\in\mathbb{R}.
\end{equation}

\end{theorem}

\begin{proof}{}
 For the reader's convenience, we proceed as in Ikeda \& Watanabe \cite{IkedaWatanabe}.

 First we prove that
 \begin{equation}\label{compa9}
  \mathbb{P}(\exists t>0; x_{1}(s)\leq x_{2}(s) \mbox{ for all } s\in [0,t])=1
 \end{equation}
under the above assumptions except that \eqref{compa4} is replaced by
\begin{equation}\label{compa4'}
 x_{1}(0)=x_{2}(0).
\end{equation}

For this, let
 \begin{equation}\label{compa10}
  \tau:=\inf\{s; b_{2}(x_{2}(s))<b_{1}(x_{1}(s))\}.
 \end{equation}
 
 For  $N\in\mathbb{N}^{\ast}$ we set
 \begin{equation}\label{compa11}
  \zeta_{N}:=\zeta_{N}^{1}\wedge\zeta_{N}^{2}
 \end{equation}
where
 \begin{equation*}
  \zeta_{N}^{1}=\inf\{t>0; |x_{1}(t)|>N\} \mbox{ and } \zeta_{N}^{2}=\inf\{t>0; |x_{2}(t)|>N\}.
 \end{equation*}

By \eqref{compa2} and \eqref{compa4'}, it is clear that $\mathbb{P}(\tau>0)=1.$ Let $t>0$ be fixed, then
 \begin{equation}\label{compa13}
  \mathbb{E}[x_{2}(t\wedge\tau)-x_{1}(t\wedge\tau)]=
  \mathbb{E}[\int_{0}^{t\wedge\tau} (\beta_{2}(s)-\beta_{1}(s))ds].
 \end{equation}

For $n\in\N^*$, let $(a_n)$ be the sequence defined by: \
$a_{0}=1>a_{1}>a_{2}>\cdots>a_{n}>\cdots\rightarrow0$ \ and satisfies, \
$$\int_{a_{n}}^{a_{n-1}} \frac{du}{u^{2}}=n.$$

For $n\in\N^*$, let $(\varphi_{n})$ be a non-negative continuous functions such that its support is contained 
in $(a_{n},a_{n-1})$ \ and satisfies,
$$\int_{a_{n}}^{a_{n-1}} \varphi_{n}(u)du=1 \ \ \ \mbox{ and } \ \ \  u^{2}\varphi_{n}(u)\leq 2/n.$$

For every $n\in\N^*$, the function $\psi_{n}(x):=\int_{0}^{|x|}\int_{0}^{y} \varphi_{n}(z)dz$ has  then the following properties,
$$\psi_{n}\in C^{2}(\mathbb{R}), \ \ \ \ \ \psi_{n}(x)\uparrow|x| \  \mbox{when} \ n\rightarrow +\infty
\ \ \ \ \ \mbox{and} \ \ \ \ \ |\psi_{n}'(x)|\leq 1.$$

For $t>0$ we set $\tilde t :=t\wedge\tau\wedge\zeta_{N}$. Using It\^o's formula and taking the expectation,
it follows that
\begin{equation}
\begin{split}
 \mathbb{E}\psi_{n}(x_{2}(\tilde{t})-x_{1}(\tilde{t}))=
 \mathbb{E}\int_{0}^{\tilde{t}} \psi_{n}'(x_{2}(s)-x_{1}(s))(\beta_{2}(s)-\beta_{1}(s))ds\\
 +\frac{1}{2}\mathbb{E}\int_{0}^{\tilde{t}}\psi_{n}"(x_{2}(s)-x_{1}(s))(\sigma (x_{2}(s))-\sigma (x_{1}(s)))^{2}ds.
\end{split}
\end{equation}

Since $\tilde{t}\leq\tau$, $\beta_{2}(u)-\beta_{1}(u) \geq b_{2}(x_{2}(u))-b_{1}(x_{1}(u))\geq0$ for all $u\leq\tilde{t}$.

Thanks to hypothesis \eqref{compa1}, we obtain:
\begin{equation}
 \begin{split}
  \mathbb{E}\psi_{n}(x_{2}(\tilde{t})-x_{1}(\tilde{t})) \leq
  \mathbb{E}\int_{0}^{\tilde{t}} \psi_{n}'(x_{2}(s)-x_{1}(s))(\beta_{2}(s)-\beta_{1}(s))ds\\
  +C\log{N}\mathbb{E}\int_{0}^{\tilde{t}}\varphi_{n}(x_{2}(s)-x_{1}(s))(x_{2}(s)-x_{1}(s))^{2}ds\\
  +C\frac{\log{N}}{N^{\mu}}\mathbb{E}\int_{0}^{\tilde{t}}\varphi_{n}(x_{2}(s)-x_{1}(s))ds.
 \end{split}
\end{equation}

Letting $n$ tends to $+\infty$ and using the fact that $|\psi_{n}'(x)|\leq1$ and $u^{2}\varphi_{n}(u)\leq 2/n$, we have
\begin{equation}
\begin{split}
  \mathbb{E}|x_{2}(t\wedge\zeta_{N}\wedge\tau)-x_{1}(t\wedge\zeta_{N}\wedge\tau)|
  &=\lim_{n\rightarrow +\infty}\mathbb{E}[\psi_{n}(x_{2}(t\wedge\zeta_{N}\wedge\tau)-x_{1}(t\wedge\zeta_{N}\wedge\tau)]\\
  &\leq \mathbb{E}\int_{0}^{t\wedge\zeta_{N}\wedge\tau} (\beta_{2}(s)-\beta_{1}(s))ds \\
  &+C\frac{\log{N}}{N^{\mu}}\limsup_{n\rightarrow +\infty}
  \mathbb{E}\int_{0}^{t\wedge\zeta_{N}\wedge\tau}\varphi_{n}(x_{2}(s)-x_{1}(s))ds.
\end{split}
\end{equation}

Since the processes are assumed to be conservative, then letting  $N$ tends to $+\infty$ and using \eqref{compa13}, it follows  that:
\begin{equation}
 \mathbb{E}|x_{2}(t\wedge\tau)-x_{1}(t\wedge\tau)|\leq \mathbb{E}[x_{2}(t\wedge\tau)-x_{1}(t\wedge\tau)].
\end{equation}

By the continuity of $x_{i}(s)$, we have
\begin{equation}
 \mathbb{P}\{t\in [0,\tau] \Rightarrow x_{1}(t) \leq x_{2}(t)\}=1
\end{equation}
and this implies \eqref{compa9}.

To prove the first assertion of the theorem, we let $\theta=\inf\{s;x_{1}(s)>x_{2}(s)\}$ and then it suffices
to show that $\theta=\infty$, almost surely.

Suppose, on the contrary, $\mathbb{P}[\theta<\infty]>0$ and set
$\tilde{\Omega}=\{\omega;\theta(\omega)<\infty\}$, $\tilde{\mathcal{F}_{t}}=\mathcal{F}_{t+\theta}|\tilde{\Omega}$,
$\tilde{\mathcal{F}}=\mathcal{F}|\tilde{\Omega}$ and $\tilde{\mathbb{P}}(A)=\mathbb{P}(A)/\mathbb{P}(\tilde{\Omega})$,
$A\in\tilde{\mathcal{F}}$.

On the space $(\tilde{\Omega},\tilde{\mathcal{F}},\tilde{\mathbb{P}},\tilde{\mathcal{F}_{t}})$,
we set $\tilde{x}_{i}(t)=x_{i}(t+\theta)$,
$\tilde{\beta}_{i}(t)=\beta_{i}(t+\theta)$, $i=1,2$, and $\tilde{B}(t)=B(t+\theta)-B(\theta)$.

Then, it is clear that $\tilde{x}_{1}(0)=x_{1}(\theta)=x_{2}(\theta)=\tilde{x}_{2}(0)$ almost surely and also,
$\tilde{\beta}_{1}(t)\leq b_{1}(\tilde{x}_{1}(t))$, $\tilde{\beta}_{2}(t)\geq b_{2}( \tilde{x}_{2}(t))$ almost surely.

Furthermore,
\begin{equation*}
 \tilde{x}_{i}(t)-\tilde{x}_{i}(0)=\int_{0}^{t} \sigma(\tilde{x}_{i}(s))d\tilde{B}(s)
 +\int_{0}^{t} \tilde{\beta}_{i}(s)ds, \, i=1,2.
\end{equation*}

Therefore, we can apply \eqref{compa9} and obtain
\begin{equation*}
 \tilde{\mathbb{P}}[\exists t>0; s\in [0,t] \Rightarrow \tilde{x}_{1}(s) \leq \tilde{x}_{2}(s)]=1.
\end{equation*}

But this contradicts with the definition of $\theta$. Therefore, $\theta=\infty$ almost surely and hence \eqref{compa} is proved.

The second assertion can be proved by similar arguments as in Ikeda \& Watanabe \cite{IkedaWatanabe}.
To be quite explicit, assume that one of the SDEs \eqref{compa7}, say for $i=1$, the pathwise
uniqueness  holds. Let $X(t)$ be the solution of the equation
\begin{equation}\label{compadeux}
X(t)=x_{1}(0)+\int_{0}^{t} \sigma(X(s))dW_{s} + \int_{0}^{t} b_{1}(X(s))ds
\end{equation}
and for $\varepsilon>0$, $X^{\pm\varepsilon}(t)$ the respective solutions of
\begin{equation}
X^{\pm\varepsilon}(t)=x_{1}(0)+\int_{0}^{t} \sigma(X(s))dW_{s} + \int_{0}^{t} [b_{1}(X(s))\pm\varepsilon]ds.
\end{equation}

Then, by the first part of the proof, we have
\begin{equation}
X^{-\varepsilon}(t)\leq X(t) \leq X^{\varepsilon}(t), \, \mbox{ for all } t\geq0.
\end{equation}

Now, noticing that $\beta_{1}(t)\leq b_{1}(x_{1}(t))$ \ $a.s.$ and $b_{1}(x)<b_{1}(x)+\varepsilon$,
we obtain thanks to the first part of the proof,
$x_{1}(t)\leq X^{\varepsilon}(t)$ and then, tending $\varepsilon$ to 0, we get $x_{1}(t)\leq X(t)$.
In a same manner, notice that $\beta_{2}(t)\geq b_{2}(x_{2}(t))$ a.s. and $b_{2}(x)\geq b_{1}(x)> b_{1}(x)-\varepsilon$.
Then, again thanks to the first part, we have $X^{-\varepsilon}(t)\leq x_{2}(t)$, and tending $\varepsilon$ to 0,
we get $X(t)\leq x_{2}(t)$. This achieves the proof.
\end{proof}

\subsection{Non-contact property}

Now, we prove the non-confluence of the solutions of the SDE \eqref{1}.

\begin{theorem}\label{non-contact}

We let $T>0$ given, we assume that the coefficients $\sigma$ and $b$ satisfy hypothesis \eqref{H1} and we assume that
the solutions of SDE \eqref{1} are conservative. For any $x,y\in\mathbb{R}^{d}$, we denote by
$X_{t}(x)$ and $X_{t}(y)$ the solutions of SDE \eqref{1} starting respectively from x and y.

Then, if $x\neq y$ we have almost surely $X_{t}(x) \neq X_{t}(y)$ for all $0\leq t\leq T$.

\end{theorem}

\begin{proof}{}
For all $\varepsilon >0$ and any real $p$, we consider the function $F(x)=f(x)^{p}$ with $f(x)=\varepsilon + {| x |}^{2}.$

We let $\tau:=\inf\{t>0; |X_{t}(x)-X_{t}(y)|^{2}=0\}$ and for any $N\in\mathbb{N}^{\ast}$, we set
$$\zeta_{N}:=\inf\{t>0; |X_{t}(x)|>N \mbox{ or } |X_{t}(y)|>N \}$$
and  $$\tau_{N}:=\inf\{t>0; |X_{t}(x)-X_{t}(y)|^{2}=\frac{1}{N^{\mu}} \}.$$
Then, as $N$ goes to $+\infty$, we have $\zeta_{N}$ tends to $+\infty$ a.s. and $\tau_{N}$ tends to $\tau$ a.s.

Set $\eta_{t}:=X_{t}(x)-X_{t}(y)$. We use It\^o formula to get
\begin{equation*}
\begin{split}
F(\eta_{t\wedge\zeta_{N}})&=F(\eta_{0}) \\
&+\int_{0}^{t\wedge\zeta_{N}}
\langle D^{1}F(\eta_{s\wedge\zeta_{N}}),
(\sigma (X_{s\wedge\zeta_{N}}(x))-\sigma (X_{s\wedge\zeta_{N}}(y)))dW_{s} \rangle \\
&+\int_{0}^{t\wedge\zeta_{N}} \langle D^{1}F(\eta_{s\wedge\zeta_{N}}),
(b(X_{s\wedge\zeta_{N}}(x))-b(X_{s\wedge\zeta_{N}}(y))) \rangle ds \\
&+ \frac{1}{2} \int_{0}^{t\wedge \zeta_{N}} \mbox{Trace}\{  D^{2}F(\eta_{s\wedge\zeta_{N}})
(\sigma (X_{s\wedge\zeta_{N}}(x))-\sigma (X_{s\wedge\zeta_{N}}(y))) \\
&(\sigma(X_{s\wedge\zeta_{N}}(x))-\sigma(X_{s\wedge\zeta_{N}}(y)))^{\top} \} ds
\end{split}
\end{equation*}
where $D^{1}F$ and $D^{2}F$ are, respectively, the gradient and the Hessian matrix of F.

Taking, respectively, the expectation and the absolute value, it follows that
\begin{equation*}
\begin{split}
\mathbb{E}[F(\eta_{t\wedge\zeta_{N}})]&\leq F(\eta_{0}) \\
& + 2 |p| \mathbb{E} \int_{0}^{t}|\eta_{s\wedge\zeta_{N}}|
{|f(\eta_{s\wedge\zeta_{N}})|}^{p-1}|b(X_{s\wedge\zeta_{N}}(x))-b(X_{s\wedge\zeta_{N}}(y))| ds \\
&+|p| \mathbb{E}\int_{0}^{t}[|f(\eta_{s\wedge\zeta_{N}})|^{p-1}+2|p-1||{\eta_{s\wedge\zeta_{N}}}|^{2}|
f(\eta_{s\wedge\zeta_{N}})|^{p-2}] \\
&{||\sigma (X_{s\wedge\zeta_{N}}(x))-\sigma (X_{s\wedge\zeta_{N}}(y))||}^{2}ds.
\end{split}
\end{equation*}

According to assumption \eqref{H1}, we have
\begin{equation*}
\begin{split}
\mathbb{E}[F(\eta_{t\wedge\zeta_{N}})]\leq F(\eta_{0})+ 2 C|p|\log{N} \mathbb{E}
\int_{0}^{t}|\eta_{s\wedge\zeta_{N}}|^{2}
{|f(\eta_{s\wedge\zeta_{N}})|}^{p-1}ds\\
+ 2C|p|\frac{\log{N}}{N^{\mu}}\mathbb{E}\int_{0}^{t} {|f(\eta_{s\wedge\zeta_{N}})|}^{p-1}ds \\
+ C|p|\frac{\log{N}}{N^{\mu}}\mathbb{E}\int_{0}^{t} |f(\eta_{s\wedge\zeta_{N}})|^{p-1}ds\\
+2C|p(p-1)|\frac{\log{N}}{N^{\mu}}\mathbb{E}\int_{0}^{t} |\eta_{s\wedge\zeta_{N}}|^{2}|
f(\eta_{s\wedge\zeta_{N}})|^{p-2}ds\\
+C|p|\log{N} \mathbb{E}\int_{0}^{t}[|\eta_{s\wedge\zeta_{N}}|^{2}|f(\eta_{s\wedge\zeta_{N}})|^{p-1}\\
+2|p-1||\eta_{s\wedge\zeta_{N}}|^{4}|f(\eta_{s\wedge\zeta_{N}})|^{p-2}]ds.
\end{split}
\end{equation*}

Since $|\eta_{s\wedge\zeta_{N}}|^{2}\leq f(\eta_{s\wedge\zeta_{N}})$, it follows that
\begin{equation}\label{hamdoull}
\begin{split}
\mathbb{E}[F(\eta_{t\wedge\zeta_{N}})]\leq F(\eta_{0})+ C(p)\log{N} \mathbb{E} \int_{0}^{t}
{|f(\eta_{s\wedge\zeta_{N}})|}^{p}ds \\
+ C(p)\frac{\log{N}}{N^{\mu}}\mathbb{E}\int_{0}^{t} {|f(\eta_{s\wedge\zeta_{N}})|}^{p-1}ds
\end{split}
\end{equation}
where $C(p)$ is a positive constant which depends only on $p$.

Since \
$|f(\eta_{s\wedge\tau_{N}\wedge\zeta_{N}})|^{-1}\leq |\eta_{s\wedge\tau_{N}\wedge\zeta_{N}}|^{-2}\leq N^{\mu}$,
it follows that
\begin{equation}\label{zouina}
\mathbb{E}[F(\eta_{t\wedge\tau_{N}\wedge\zeta_{N}})]\leq F(\eta_{0})+ 2C(p)\log{N}
\int_{0}^{t} \mathbb{E}[F(\eta_{s\wedge\tau_{N}\wedge\zeta_{N}})]ds.
\end{equation}

Using Gronwall's lemma, we obtain
\begin{equation}\label{zouina0}
 \mathbb{E}[F(\eta_{t\wedge\tau_{N}\wedge\zeta_{N}})] \leq F(\eta_{0})N^{2C(p)t}
\end{equation}
 that is
\begin{equation}\label{zouina1}
 \mathbb{E}[(\varepsilon+|X_{t\wedge\tau_{N}\wedge\zeta_{N}}(x)-X_{t\wedge\tau_{N}\wedge\zeta_{N}}(y)|^{2})^{p}]
 \leq (\varepsilon + |x-y|^{2})^{p}N^{2C(p)t}.
\end{equation}

Letting $\varepsilon$ tends to 0 in the previous inequality  and we get
\begin{equation}\label{zouina2}
 \mathbb{E}[|X_{t\wedge\tau_{N}\wedge\zeta_{N}}(x)-X_{t\wedge\tau_{N}\wedge\zeta_{N}}(y)|^{2p}]
 \leq |x-y|^{2p}N^{C(p)t}.
\end{equation}

On the subset $\{\tau_{N}\leq t\wedge\zeta_{N}\}$, we have
\begin{equation}\label{meziane}
 \frac{1}{N^{p\mu}}\mathbb{P}[\tau_{N}\leq t\wedge\zeta_{N}]\leq |x-y|^{2p}N^{2C(p)t}.
\end{equation}

Taking $p=-1$ in the previous inequality, we get
\begin{equation}\label{meziane2}
 \mathbb{P}[\tau_{N}\leq t\wedge\zeta_{N}]\leq |x-y|^{-2}N^{Ct-\mu}.
\end{equation}

Letting  $N$ tends  to $+\infty$ in the previous inequality, we obtain for $t<\mu/C$,
\begin{equation}\label{meziane3}
 \mathbb{P}[\tau\leq t]=0.
\end{equation}

Starting now from $\mu/C$ and using the same arguments as above, we get for any $t\in [\mu/C;2\mu/C[$,
\begin{equation*}
 \mathbb{P}[\tau\leq t]=0.
\end{equation*}

The sequence $T_{k}:=k\mu/C$ goes to $+\infty$ as $k$ tends $+\infty$.

Arguing recursively on $k$, one can  prove that for any $t\in [T_{k}, T_{k+1}[$
\begin{equation*}
 \mathbb{P}[\tau\leq t]=0.
\end{equation*}

This shows that, for any $t\geq 0$
\begin{equation}\label{mezianelast}
 \mathbb{P}[\tau\leq t]=0
\end{equation}
which implies that $\tau = \infty$ \ $a.s$.  The theorem is proved.
\end{proof}

\subsection{Continuous dependence}

In what follows, we prove that the solution of the SDE \eqref{1} has a continuous modification in $(t,x)$.

\begin{lemma} \label{lem 3.1}

Assume that the coefficients $\sigma$ and $b$ are bounded and satisfy hypothesis $\eqref{H1}$.
Then, for any $R, T>0$ and each $p\geq 1$, there exists a positive constant $C_{p,R,T}$ such that 
for any $|x|\leq R$, $|y|\leq R$ and any $s,t\in [0,T]$,
\begin{equation}\label{211}
 \mathbb{E}[|X_{t}(x)-X_{s}(y)|^{2p}] \leq C_{p,R,T}(|t-s|^{p}+|x-y|^{2p}+|x-y|^{p/2}+|x-y|^{5p/2})
\end{equation}
\end{lemma}

\begin{proof}{}
In the following we keep the same notations and arguments as in the proof of the
non-contact property (Theorem \ref{non-contact}).

Set $f_{N}(x):=\varepsilon + {|x|}^{2} + \frac{1}{N^{\mu}}$ and $F_{N}(x)=f_{N}(x)^{p}$.
Then, by similar arguments as in proof of non-contact property, we have the following inequality which is similar to
\eqref{hamdoull} with $F$ and $f$ replaced by $F_{N}$ and $f_{N}$:
\begin{equation}\label{hamdoullboy}
\begin{split}
\mathbb{E}[F_{N}(\eta_{t\wedge\zeta_{N}})]\leq F_{N}(\eta_{0})+ C(p)\log{N} \mathbb{E} \int_{0}^{t}
{|f_{N}(\eta_{s\wedge\zeta_{N}})|}^{p}ds \\
+ C(p)\frac{\log{N}}{N^{\mu}}\mathbb{E}\int_{0}^{t} {|f_{N}(\eta_{s\wedge\zeta_{N}})|}^{p-1}ds.
\end{split}
\end{equation}

Since $f^{-1}_{N}(x)\leq N^{\mu}$, then we have
\begin{equation}\label{hamdoullboy2}
\mathbb{E}[F_{N}(\eta_{t\wedge\zeta_{N}})]\leq F_{N}(\eta_{0})+ 2C(p)\log{N} \int_{0}^{t}
\mathbb{E}[F_{N}(\eta_{s\wedge\zeta_{N}})]ds.
\end{equation}

Thanks to the Gronwall's lemma, it follows that
\begin{equation}\label{11copy}
\mathbb{E}[F_{N}(\eta_{t\wedge\zeta_{N}})]\leq F_{N}(\eta_{0})N^{2C(p)t}
\end{equation}
and that is
\begin{equation}\label{11copy2}
 \mathbb{E}(\varepsilon +\frac{1}{N^{\mu}} + |X_{t\wedge\zeta_{N}}(x)-X_{t\wedge\zeta_{N}}(y)|^{2})^{p}
 \leq (\varepsilon +\frac{1}{N^{\mu}} + |x-y|^{2})^{p}N^{2C(p)t}.
\end{equation}

For  $T>0$, we set $Y_{T}(x)=\sup_{t\leq T} |X_{t}(x)|$.
We consider a family of smooth functions
\\
 $\varphi_{N}:\mathbb{R}^{d}\longmapsto\mathbb{R}$ satisfying
$$0\leq\varphi_{N}\leq1, \ \ \
\varphi_{N}(x)=1 \mbox{ for } |x|\leq N, \ \ \ and \ \ \  \varphi_{N}(x)=0 \mbox{ for } |x|>N+1.$$

Define $\sigma_{N}(x):=\varphi_{N}(x)\sigma(x)$, \  $b_{N}(x):=\varphi_{N}(x)b(x)$.  
Let $(X_{t}^{N}(x))$ be the solution of the SDE
\begin{equation}
 X^{N}_{t}=x+\int_{0}^{t} \sigma_{N}(X^{N}_{s})dW_{s}+\int_{0}^{t}b_{N}(X^{N}_{s})ds.
\end{equation}

Arguing as in  \cite{Globalflow}, we show that,
\begin{equation*}
(\varepsilon + {| X_{t}(x)-X_{t}(y) |}^{2})^{p}=\sum_{N=1}^{+\infty}
(\varepsilon + {| X^{N}_{t}(x)-X^{N}_{t}(y) |}^{2})^{p} 1_{\{N-1\leq Y_{T}(x)\vee Y_{T}(y) <N\}},
\end{equation*}
which implies that,
\begin{equation*}
\begin{split}
(\varepsilon &+ {| X_{t}(x)-X_{t}(y) |}^{2})^{p}=\\
&\sum_{N=1}^{+\infty}
(\varepsilon + {| X^{N}_{t}(x)-X^{N}_{t}(y) |}^{2})^{p}
(1_{\{\tau_{N}\leq T\wedge\zeta_{N}\}}+1_{\{\tau_{N} > T\wedge\zeta_{N}\}})
1_{\{N-1\leq Y_{T}(x)\vee Y_{T}(y) <N\}}.
\end{split}
\end{equation*}

Thanks to the pathwise uniqueness, it follows that
\begin{equation*}
\begin{split}
(\varepsilon &+ {| X_{t}(x)-X_{t}(y) |}^{2})^{p}=\\
&+\sum_{N=1}^{+\infty}
(\varepsilon + {| X_{t\wedge\zeta_{N}}(x)-X_{t\wedge\zeta_{N}}(y) |}^{2})^{p}
1_{\{N-1\leq Y_{T}(x)\vee Y_{T}(y) <N\}}\times 1_{\{\tau_{N}\leq T\wedge\zeta_{N}\}}\\
&+\sum_{N=1}^{+\infty} (\varepsilon +
{| X_{t\wedge\zeta_{N}\wedge\tau_{N}}(x)-X_{t\wedge\zeta_{N}\wedge\tau_{N}}(y) |}^{2})^{p}
1_{\{N-1\leq Y_{T}(x)\vee Y_{T}(y) <N\}}\times 1_{\{\tau_{N} > T\wedge\zeta_{N}\}}.
\end{split}
\end{equation*}

Taking the expection in the above inequality and thanks to Cauchy-Schwartz inequality, we get
\begin{equation*}
\begin{split}
\mathbb{E}[(\varepsilon + {|X_{t}(x)-X_{t}(y)|}^{2})^{p}]\leq \sum_{N=1}^{+\infty}
(\mathbb{E}[(\varepsilon + {| X_{t\wedge\zeta_{N}}(x)-X_{t\wedge\zeta_{N}}(y) |}^{2})^{2p}])^{1/2}
\\
\times (\mathbb{P}[N-1 \leq Y_{T}(x)\vee Y_{T}(y)])^{1/4}(\mathbb{P}[\tau_{N}\leq T\wedge\zeta_{N}])^{1/4} \\
+
\sum_{N=1}^{+\infty}
(\mathbb{E}[(\varepsilon + {| X_{t\wedge\tau_{N}\wedge\zeta_{N}}(x)-X_{t\wedge\tau_{N}\wedge\zeta_{N}}(y) |}^{2})^{2p}])^{1/2}
\\
\times (\mathbb{P}[N-1\leq Y_{T}(x)\vee Y_{T}(y)])^{1/4}(\mathbb{P}[\tau_{N} >T\wedge\zeta_{N}])^{1/4}.
\end{split}
\end{equation*}

Since \
$(\varepsilon + {| X_{t\wedge\zeta_{N}}(x)-X_{t\wedge\zeta_{N}}(y) |}^{2})
\leq
(\varepsilon + \frac{1}{N^{\mu}} + {| X_{t\wedge\zeta_{N}}(x)-X_{t\wedge\zeta_{N}}(y) |}^{2})$
and $\mathbb{P}[\tau_{N} >T\wedge\zeta_{N}]\leq 1$, it follows thanks to \eqref{meziane} that:
\begin{equation*}
\begin{split}
\mathbb{E}[(\varepsilon + {|X_{t}(x)-X_{t}(y)|}^{2})^{p}]\leq \sum_{N=1}^{+\infty}
(\mathbb{E}[(\varepsilon + \frac{1}{N^{\mu}} + {| X_{t\wedge\zeta_{N}}(x)-X_{t\wedge\zeta_{N}}(y) |}^{2})^{2p}])^{1/2}
\\
\times (\mathbb{P}[N-1 \leq Y_{T}(x)\vee Y_{T}(y)])^{1/4}|x-y|^{p/2} N^{(2C(p)T+p\mu)/4} \\
+
\sum_{N=1}^{+\infty}
(\mathbb{E}[(\varepsilon + {| X_{t\wedge\tau_{N}\wedge\zeta_{N}}(x)-X_{t\wedge\tau_{N}\wedge\zeta_{N}}(y) |}^{2})^{2p}])^{1/2}
\\
\times (\mathbb{P}[N-1\leq Y_{T}(x)\vee Y_{T}(y)])^{1/4}.
\end{split}
\end{equation*}

For $x,y\in\mathbb{R}^{d}$, let $R>0$ be such that $x,y\in \mathbb{B}(R)$.
Since the coefficients are assumed to be bounded, then arguing as in Corollary 1.2 of \cite{Globalflow}, one can show 
that there exists $\delta_{R}>0$ such that
$$\sup_{|x|\leq R}\mathbb{E}[e^{\delta_{R}Y^{2}_{T}(x)}]<+\infty.$$

Before, we continue our proof, let us recall that thanks to inequality \eqref{zouina1} we have
\begin{equation}\label{11}
 \mathbb{E}[(\varepsilon+|X_{t\wedge\tau_{N}\wedge\zeta_{N}}(x)-X_{t\wedge\tau_{N}\wedge\zeta_{N}}(y)|^{2})^{p}]
 \leq (\varepsilon + |x-y|^{2})^{p}N^{2C(p)t}.
\end{equation}

Now, we use Markov's inequality, inequality \eqref{11copy2} and inequality \eqref{11} to get
\begin{equation*}
\begin{split}
\mathbb{E}[(\varepsilon + {|X_{t}(x)-X_{t}(y)|}^{2})^{p}]\leq \sum_{N=1}^{+\infty}
(\varepsilon +\frac{1}{N^{\mu}} + |x-y|^{2})^{p}N^{2 C(p)t}
\\
\times ({\sup_{|x|\leq R}\mathbb{E}[e^{\delta_{R}Y^{2}_{T}(x)}]})^{1/4}
e^{-\frac{\delta_{R}}{4}(N-1)^{2}}
|x-y|^{p/2} N^{(2C(p)T+p\mu)/4} \\
+
\sum_{N=1}^{+\infty}
(\varepsilon + |x-y|^{2})^{p}N^{2C(p)t}
\times ({\sup_{|x|\leq R}\mathbb{E}[e^{\delta_{R}Y^{2}_{T}(x)}]})^{1/4}
e^{-\frac{\delta_{R}}{4}(N-1)^{2}}.
\end{split}
\end{equation*}

This implies that
\begin{equation}\label{seriesconverges}
\begin{split}
\mathbb{E}[(\varepsilon &+ {|X_{t}(x)-X_{t}(y)|}^{2})^{p}]\leq \\
&2^{p-1}(\varepsilon + |x-y|^{2})^{p}|x-y|^{p/2}({\sup_{|x|\leq R}\mathbb{E}[e^{\delta_{R}Y^{2}_{T}(x)}]})^{1/4}
\sum_{N=1}^{+\infty} N^{C(\mu, p, T)}
e^{-\frac{\delta_{R}}{4}(N-1)^{2}}
\\
&+
2^{p-1}|x-y|^{p/2}({\sup_{|x|\leq R}\mathbb{E}[e^{\delta_{R}Y^{2}_{T}(x)}]})^{1/4}\sum_{N=1}^{+\infty}N^{C(\mu, p, T)}
e^{-\frac{\delta_{R}}{4}(N-1)^{2}} \\
&+
2^{p-1}(\varepsilon + |x-y|^{2})^{p} ({\sup_{|x|\leq R}\mathbb{E}[e^{\delta_{R}Y^{2}_{T}(x)}]})^{1/4}
\sum_{N=1}^{+\infty} N^{C(\mu, p, T)} e^{-\frac{\delta_{R}}{4}(N-1)^{2}}.
\end{split}
\end{equation}

Since the series in the right-hand side of \eqref{seriesconverges} converges,
there is a positive constant $C_{p,R,T}$ such that:
\begin{equation}\label{14}
\mathbb{E}[(\varepsilon + {| X_{t}(x)-X_{t}(y) |}^{2})^{p}]\leq
C_{p,R,T}[(\varepsilon + |x-y|^{2})^{p}+|x-y|^{p/2}+(\varepsilon + |x-y|^{2})^{p}|x-y|^{p/2}].
\end{equation}

Letting $\varepsilon$ tends to 0 in \eqref{14}, we get for all $t\in [0,T]$
\begin{equation}\label{15}
\mathbb{E}[{| X_{t}(x)-X_{t}(y) |}^{2p}]\leq C_{p,R,T}(|x-y|^{2p}+|x-y|^{p/2}+|x-y|^{5p/2}).
\end{equation}

In addition, since $\sigma$ and $b$ are assumed to be bounded, it is enough to prove using Burkholder
and H\"older inequalities that there exists a positive constant $C_{p,T}$ such that for any $s,t\in [0,T]$ we have
\begin{equation}\label{16}
\mathbb{E}[{|X_{t}(x)-X_{s}(x)|}^{2p}]\leq C_{p,T}|t-s|^{p}.
\end{equation}

Combining \eqref{15} and \eqref{16}, it follows that for any $s,t\in [0,T]$ and any $|x|,|y|\leq R$,
there exists a constant $C_{p,R,T}$ such that
\begin{equation*}
 \mathbb{E}[|X_{t}(x)-X_{s}(y)|^{2p}] \leq C_{p,R,T}(|t-s|^{p}+|x-y|^{2p}+|x-y|^{p/2}+|x-y|^{5p/2}).
\end{equation*}
The proof is finished.
\end{proof}

\begin{theorem} \label{bicontinuite}

Assume that assumption \eqref{H1} holds and the SDE \eqref{1} is strictly conservative.
Then, the solution of SDE \eqref{1} admits a version which is bi-continuous in  $(t,x) \ a.s$.

\end{theorem}

\begin{proof}{}
We shall split the proof into two steps.

$Step \ 1$. We assume that the coefficients $\sigma$ and $b$ are compactly supported.
Then, thanks to Lemma \ref{lem 3.1}, for any $T>0$ and any $p>1$, we have for any $|x|,|y|\leq R$ and all $t,s \in [0,T]$,
\begin{equation*}
  \mathbb{E}[|X_{t}(x)-X_{s}(y)|^{2p}] \leq C_{p,R,T}(|t-s|^{p}+|x-y|^{2p}+|x-y|^{p/2}+|x-y|^{5p/2})
\end{equation*}
Taking $p>d+1$ and using the Kolmogorov theorem, we show that the solution $X_{t}(x)$
of SDE \eqref{1} admits a version [denoted by $\tilde{X}_{t}(x)$] which is continuous on $[0,T]\times\{|x| \leq R\}$ \ $a.s$. \

In addition, since $\sigma$, $b$ are with compact support and the pathwise uniqueness
holds for the SDE \eqref{1}, it is easy to prove that if $|x| \leq R$, then $|X_{t}(x)|\leq R$ for any $t\in[0,T]$. Then,
if $|x|>R$, we have $X_{t}(x)=x$ for any $t\in[0,T]$. Thus, we can extend continuously $(t,x)\mapsto\tilde{X}_{t}(x)$
on $[0,T]\times\mathbb{R}^{d}$ \ $a.s$.

For $0<t\leq T$ and $\omega(t)\in\mathbb{R}^{m}$, \ let $\theta_{T}(\omega)(t):=\omega(t+T)-\omega(T)$ and
$$\tilde{X}_{T+t}(x,\omega):=\tilde{X}_{t}(\tilde{X}_{T}(x,\omega), \theta_{T}(\omega)).$$
The process $\tilde{X}_{T+t}(x)$ satisfies then the SDE \eqref{1}.

By pathwise uniqueness it follows that for every  $t\in[0,T]$, \ $\tilde{X}_{T+t}(x)=X_{T+t}(x)$ \ $a.s$.
This means that $\tilde{X}_{t}(x)$ is a continuous version of $X_{t}(x)$ on $[0,2T]\times\mathbb{R}^{d}$.
Reasoning successively in this way, we show that $\tilde{X}_{t}(x)$ is a continuous version of $X_{t}(x)$ on the whole space
$\mathbb{R}_{+}\times\mathbb{R}^{d}$.

$Step \ 2.$ \  The coefficients $\sigma, b$ are not compactly supported. We will proceed as in \cite{Zhang} who themselves have proceed as in Protter \cite{Protter}.
Precisely, for any $R>0$ we consider a smooth function with compact support $\varphi_{R}:\mathbb{R}^{d}\rightarrow\mathbb{R}$
satisfying
$$0\leq\varphi_{R}\leq1, \ \ \ \ \varphi_{R}(x)=1 \ \mbox{ for } |x|\leq R \ \ \ and \ \ \  \varphi_{R}(x)=0 \ \mbox{ for } |x|>R+1.$$

We put $\sigma_{R}(x):=\varphi_{R}(x)\sigma(x)$ and $b_{R}(x):=\varphi_{R}(x)b(x)$.  Let $(X_{t}^{R}(x))$ be the solution of SDE \eqref{1} with $\sigma$ and b replaced by $\sigma_{R}$ and $b_{R}$.
According to the first step of this proof, let $\tilde{X}_{t}^{R}(x)$ be a continuous version of $X_{t}^{R}(x)$.
For $K>0$, we set
$$\zeta_{K}^{R}(x):=\inf\{t>0; |\tilde{X}_{t}^{R}(x)|\geq K\} \mbox{ and } \zeta_{K}=\inf\{t>0; |X_{t}(x)|\geq K\}.$$

Since the pathwise uniqueness holds, for $|x|\leq R$,
$$X_{t}(x)=\tilde{X}_{t}^{N}(x) \mbox{ for any } N>R+1 \mbox{ and } t<\zeta_{R+1}^{N},$$
or
$$\zeta_{R+1}(x)=\zeta^{N}_{R+1}(x) \mbox{ for any } N>R+1.$$

For $|x|\leq R$, we define
$$\tilde{X}_{t}(x)=\tilde{X}_{t}^{R+2}(x) \mbox{ on } [0, \zeta_{R+1}^{R+2}(x)[.$$
Then $\tilde{X}_{.}(x)$ is a version of $X_{.}(x)$.

Let us prove that $\tilde{X}_{t}(x)$ is continuous in $(t,x)$ almost everywhere.
Fix $x_{0}$ with $|x_{0}|\leq R$. By the strict conservativeness of the SDE \eqref{1}, there exists $R>0$ such that
$\zeta^{R+2}_{R+1}(x_{0})>t+\varepsilon$ for any small strictly positive real $\varepsilon$.
This implies that $\sup_{0\leq s\leq t+\varepsilon} |\tilde{X}_{s}^{R+2}(x_{0})|<R+1$.
By continuity, we can find a neighbourhood $B_{\delta}(x_{0})$ of $x_{0}$ such that
$\sup_{0\leq s\leq t+\varepsilon} |\tilde{X}_{s}^{R+2}(x)|<R+1$ or $\tau^{R+2}_{R+1}(x)>t+\varepsilon$
for all $x\in B_{\delta}(x_{0})$. Hence, almost everywhere $\tilde{X}_{s}(x)=\tilde{X}_{s}^{R+2}(x)$
for all $x\in B_{\delta}(x_{0})$ and $s\leq t+\varepsilon$, which implies that $\tilde{X}_{s}(x_{0})$
is continuous at the point $(t,x_{0})$.  Theorem \ref{bicontinuite} is proved.
\end{proof}

\section{Large deviations of Freidlin-Wentzell type}

The main task of this section, is to prove a large deviations principle of Freidlin-Wentzell type under assumption \eqref{H1}.
For this, dealing with Euler scheme, we establish two key lemmas for exponential tightness and contraction principle, and we conclude
thanks to a result of \cite{Dembozeitouni}.
First, let us prove an Euler scheme for the unique strong solution of  SDE \eqref{1}.

\subsection{Euler scheme}

We recall the following classical estimate for stochastic integrals which can be proved by exponential martingales, 
see for instance \cite{Stroock}.

\begin{lemma}\label{estimate}

Let e and f be respectively matrix-valued and vector-valued adapted processes.
Assume that they are bounded i.e., $||e_{t}(\omega)||\leq A$ and $|f_{t}(\omega)|\leq B$ for all $(t,\omega)$ and
consider the following It\^o process on $\mathbb{R}^{d}$
\begin{equation*}
 \eta_{t}=\eta_{0}+\int_{0}^{t} e_{s}dW_{s}+\int_{0}^{t} f_{s}ds, \ \ \ \ \mbox{where} \  \eta_{0}\in\mathbb{R}^{d}.
\end{equation*}
Then, for any $T>0$ and $R>\sqrt{d}BT$, we have
\begin{equation*}
\mathbb{P}(\sup_{0\leq t\leq T} |\eta_{t}|^{2}\geq R)\leq 2de^{-(R-\sqrt{d}BT)^{2}/2dA^{2}T}.
\end{equation*}
\end{lemma}

The following result could be deduced thanks to \cite{BahlaliOuknineMezerdi}. 
However, for the reader convenience and for our need, we will prove it here by a different method.

\begin{theorem}\label{approx}
Assume that the coefficients $\sigma$ and $b$ are bounded and satisfy assumption \eqref{H1}.

For $n\in\mathbb{N}^{\ast}$, define $(X_{n}(t))_{n\geq 1}$ by,  \
$X_{n}(0):=x$ \ and \ for $t\in [k2^{-n}; (k+1)2^{-n}[$,
\begin{equation*}
X_{n}(t):=X_{n}(k2^{-n})+\sigma (X_{n}(k2^{-n}))(W(t)-W(k2^{-n}))+b(X_{n}(k2^{-n}))(t-k2^{-n}).
\end{equation*}

Then, $X_{n}(t)$ converges (in the $L^{2}$ sense) to the unique solution $X(t)$ of
the SDE \eqref{1}, that is for any $T>0$,
\begin{equation}
\lim_{n\rightarrow +\infty}\mathbb{E}\sup_{t\leq T} |X(t)-X_{n}(t)|^{2}=0.
\end{equation}

\end{theorem}

\begin{proof}{}
Set $\phi_{n}(t):=k2^{-n}$ for $t\in [k2^{-n}, (k+1)2^{-n}[$, $k\geq0$. Then, $X_{n}(t)$ can be expressed by
\begin{equation*}
X_{n}(t)=x+\int_{0}^{t} \sigma (X_{n}(\phi_{n}(s)))dW(s)+\int_{0}^{t} b(X_{n}(\phi_{n}(s)))ds.
\end{equation*}
Let $1<a<\sqrt{2}$ \ \ and \  \
$\tau_{n}:=\inf\{t>0: |X_{n}(t)-X_{n}(\phi_{n}(t))|\geq a^{-n}\}$.

For $t\in [k2^{-n}, (k+1)2^{-n}[$, we have
\begin{equation*}
X_{n}(t)-X_{n}(\phi_{n}(t))=\int_{0}^{t-k2^{-n}}\sigma (X_{n}(\phi_{n}(s+k2^{-n})))d\tilde{W}(s)
+\int_{0}^{t-k2^{-n}} b(X_{n}(\phi_{n}(s+k2^{-n})))ds
\end{equation*}
where $\tilde{W}(s)=W(k2^{-n}+s)-W(s)$.

Using Lemma \ref{estimate}, it follows that
\begin{equation*}
 \mathbb{P} \left(\sup_{k2^{-n}\leq t < (k+1)2^{-n}} |X_{n}(t)-X_{n}(\phi_{n}(t))|\geq a^{-n} \right) \leq \\
 2d \exp\{ -(a^{-n}-\sqrt{d}B2^{-n})^{2}/2dA^{2}2^{-n}\}
\end{equation*}
where A and B are respectively the uniform bound on $\sigma$ and $b$.

We set $c=2/a^{2}$ and then for large n, we get
\begin{equation*}
 \mathbb{P} \left(\sup_{k2^{-n}\leq t < (k+1)2^{-n}} |X_{n}(t)-X_{n}(\phi_{n}(t))|\geq a^{-n} \right)
 \leq 2d \exp\{ -c^{n}/4dA^{2}\}
\end{equation*}
and for any integer $T>0$,
\begin{equation*}
 \mathbb{P} \left(\sup_{0\leq t \leq T} |X_{n}(t)-X_{n}(\phi_{n}(t))|\geq a^{-n} \right) \leq
 2d2^{n}Texp\{ -c^{n}/4dA^{2}\}.
\end{equation*}

This implies that
\begin{equation} \label{estimates123}
 \mathbb{P} (\tau_{n}\leq T) \leq 2^{n+1}dTexp\{ -c^{n}/4dA^{2}\}.
\end{equation}

Following \cite{Heunis},  we define
for any $N\in\mathbb{N}^{\ast}$,
\begin{equation*}
\zeta^{n}_{N}:=\inf\{t>0; |X(t)|>N \mbox{ or } |X_{n}(t)|>N \}.
\end{equation*}

Clearly, for each $N$ and $n$ ,
\begin{equation}\label{continuity}
\begin{split}
\mathbb{E}\sup_{t \leq T} {|X(t\wedge\tau_{n})-X_{n}(t\wedge\tau_{n})|}^{2}
\leq 3\mathbb{E}\sup_{t \leq T} {|X(t\wedge\tau_{n})-X(t\wedge\tau_{n}\wedge\zeta^{n}_{N})|}^{2}\\
+3\mathbb{E}\sup_{t \leq T} {|X(t\wedge\tau_{n}\wedge\zeta^{n}_{N})-X_{n}(t\wedge\tau_{n}\wedge\zeta^{n}_{N})|}^{2} \\
+3\mathbb{E}\sup_{t \leq T} {|X_{n}(t\wedge\tau_{n}\wedge\zeta^{n}_{N})-X_{n}(t\wedge\tau_{n})|}^{2}.
\end{split}
\end{equation}

Since the coefficients $\sigma$ and $b$ are bounded, it follows that
\begin{equation}
 \lim_{N\rightarrow +\infty} \mathbb{P}[\zeta^{n}_{N}>T]=1 \ \ \mbox{ uniformly with respect to } n.
\end{equation}

For a fixed $\gamma>0$, let $N(\gamma)\in \N^*$ be such that $N(\gamma)>Te^{1/\gamma}$ and for every $n\in \N^*$,
$$\mathbb{P}[\zeta^{n}_{N(\gamma)}>T]>(1-\gamma).$$

The first term in the right hand side of \eqref{continuity} can be estimated  as follows for  $N=N(\gamma)$:
\begin{equation}\label{continuity1}
\begin{split}
\mathbb{E}\sup_{t \leq T} {|X(t\wedge\tau_{n})-X(t\wedge\tau_{n}\wedge\zeta^{n}_{N(\gamma)})|}^{2}
\leq 2\mathbb{E}\sup_{t\leq T}|\int_{t\wedge\tau_{n}\wedge\zeta^{n}_{N(\gamma)}}^{t\wedge\tau_{n}} b(X(s))ds|^{2}\\
 +2\mathbb{E}\sup_{t\leq T}|\int_{t\wedge\tau_{n}\wedge\zeta^{n}_{N(\gamma)}}^{t\wedge\tau_{n}} \sigma (X(s))dW(s)|^{2}.
\end{split}
\end{equation}

The first term on the right hand side of \eqref{continuity1} can be estimated as follows
\begin{equation*}
 \sup_{n}\mathbb{E}\sup_{t\leq T}|
 \int_{t\wedge\tau_{n}\wedge\zeta^{n}_{N(\gamma)}}^{t\wedge\tau_{n}} b(X(s))ds|^{2} \leq TM^{2}\gamma
\end{equation*}
 while for the second term, Doob's inequality gives
\begin{equation*}
 \begin{split}
  \sup_{n}\mathbb{E}\sup_{t\leq T}|\int_{t\wedge\tau_{n}\wedge\zeta^{n}_{N(\gamma)}}^{t\wedge\tau_{n}} \sigma (X(s))dW(s)|^{2}
  &= \mathbb{E}\sup_{t\leq T}|\int_{0}^{t\wedge\tau_{n}} \chi_{[\zeta^{n}_{N(\gamma)},\infty)}(s)\sigma (X(s))dW(s)|^{2}\\
  &\leq 4 \mathbb{E}\int_{0}^{T} \chi_{[\zeta^{n}_{N(\gamma)},\infty)}(s) Trace\{\sigma \sigma^{T}(X(s))\}ds\\
  &\leq 4TM^{2}\gamma.
 \end{split}
\end{equation*}

Whence the first term on the right of \eqref{continuity} can be bounded from above as follows:
\begin{equation}\label{continuity2}
 \sup_{n}\mathbb{E}\sup_{t \leq T} {|X(t\wedge\tau_{n})-X(t\wedge\tau_{n}\wedge\zeta^{n}_{N(\gamma)})|}^{2}\leq 10TM^{2}\gamma.
\end{equation}

In the same way, the following holds for the third term on the right of \eqref{continuity} [with $N=N(\gamma)$]
\begin{equation}\label{continuity3}
 \mathbb{E}\sup_{t \leq T} {|X_{n}(t\wedge\tau_{n}\wedge\zeta^{n}_{N(\gamma)})-X_{n}(t\wedge\tau_{n})|}^{2}
 \leq 10TM^{2}\gamma.
\end{equation}

We shall estimate the second term in the right hand side of \eqref{continuity} [again with $N=N(\gamma)$]. Using It\^o formula,
we get
\begin{equation*}
\begin{split}
{|X(t\wedge\tau_{n}\wedge\zeta^{n}_{N(\gamma)})-X_{n}(t\wedge\tau_{n}\wedge\zeta^{n}_{N(\gamma)})|}^{2}
=\int_{0}^{t\wedge\tau_{n}\wedge\zeta^{n}_{N(\gamma)}} ||\sigma (X(s))-\sigma (X_{n}(\phi_{n}(s)))||^{2}ds\\
+2\int_{0}^{t\wedge\tau_{n}\wedge\zeta^{n}_{N(\gamma)}}
\langle (X(s)-X_{n}(s)), (\sigma(X(s))-\sigma(X_{n}(\phi_{n}(s))))dW(s) \rangle\\
+2\int_{0}^{t\wedge\tau_{n}\wedge\zeta^{n}_{N(\gamma)}} \langle (X(s)-X_{n}(s)), (b(X(s))-b(X_{n}(\phi_{n}(s)))) \rangle ds.\\
\end{split}
\end{equation*}

For any $t\geq0$,  we set $\tilde{t}=t\wedge\tau_{n}\wedge\zeta^{n}_{N(\gamma)}$.

Thanks to Burkholder's inequality, we get
\begin{equation*}
\begin{split}
\mathbb{E}[\sup_{t\leq T} |X(\tilde{t})-X_{n}(\tilde{t})|^{2}] \leq
\mathbb{E}\int_{0}^{T} ||\sigma (X(\tilde{s}))-\sigma (X_{n}(\phi_{n}(\tilde{s})))||^{2}ds\\
+2 C_{1} \mathbb{E}{\left( \int_{0}^{T} ||\sigma(X_{n}(\phi_{n}(\tilde{s})))-\sigma(X(\tilde{s}))||^{2}
|X(\tilde{s})-X_{n}(\tilde{s})|^{2}ds\right)}^{1/2}\\
+2\mathbb{E}\int_{0}^{T} |b(X(\tilde{s}))-b(X_{n}(\phi_{n}(\tilde{s})))||X(\tilde{s})-X_{n}(\tilde{s})|ds.
\end{split}
\end{equation*}

This implies that
\begin{equation*}
\begin{split}
 \mathbb{E}[\sup_{t\leq T} |X(\tilde{t})-X_{n}(\tilde{t})|^{2}]
 &\leq \frac{1}{2}\mathbb{E}[\sup_{t\leq T} |X(\tilde{t})-X_{n}(\tilde{t})|^{2}]\\
&+2\mathbb{E}\int_{0}^{T} |b(X(\tilde{s}))-b(X_{n}(\phi_{n}(\tilde{s})))||X_{n}(\tilde{s})-X(\tilde{s})|ds\\
&+(1+2C_{1}^{2})\mathbb{E}\int_{0}^{T} ||\sigma (X(\tilde{s}))-\sigma (X_{n}(\phi_{n}(\tilde{s})))||^{2}ds.
\end{split}
\end{equation*}

Using triangular inequalities, it follows that
\begin{equation*}
\begin{split}
\mathbb{E}[\sup_{t\leq T} |X(\tilde{t})-X_{n}(\tilde{t})|^{2}] &\leq
4(1+2C_{1}^{2})\mathbb{E}\int_{0}^{T} ||\sigma (X(\tilde{s}))-\sigma (X_{n}(\tilde{s}))||^{2}ds\\
&+4(1+2C_{1}^{2})\mathbb{E}\int_{0}^{T} ||\sigma (X_{n}(\tilde{s}))-\sigma (X_{n}(\phi_{n}(\tilde{s})))||^{2}ds\\
&+4\mathbb{E}\int_{0}^{T} |b(X(\tilde{s}))-b(X_{n}(\tilde{s}))||X(\tilde{s})-X_{n}(\tilde{s})|ds\\
&+4\mathbb{E}\int_{0}^{T} |b(X_{n}(\tilde{s}))-b(X_{n}(\phi_{n}(\tilde{s})))||X(\tilde{s})-X_{n}(\tilde{s})|ds.
\end{split}
\end{equation*}

Thanks to Cauchy-Schwartz's inequality, we get
\begin{equation*}
\begin{split}
\mathbb{E}[\sup_{t\leq T} |X(\tilde{t})-X_{n}(\tilde{t})|^{2}] \leq
4(1+2C_{1}^{2})\mathbb{E}\int_{0}^{T} ||\sigma (X_{n}(\tilde{s}))-\sigma (X(\tilde{s}))||^{2}ds\\
+4(1+2C_{1}^{2})\mathbb{E}\int_{0}^{T} ||\sigma (X_{n}(\phi_{n}(\tilde{s})))-\sigma (X_{n}(\tilde{s}))||^{2}ds\\
+4\mathbb{E}[{\left( \int_{0}^{T} |b(X_{n}(\phi_{n}(\tilde{s})))-b(X_{n}(\tilde{s}))|^{2}ds \right)}^{1/2}
{\left( \int_{0}^{T} |X_{n}(\tilde{s})-X(\tilde{s})|^{2}ds \right)}^{1/2}]\\
+4\mathbb{E}\int_{0}^{T} |b(X_{n}(\tilde{s}))-b(X(\tilde{s}))||X_{n}(\tilde{s})-X(\tilde{s})|ds.
\end{split}
\end{equation*}

This implies that
\begin{equation*}
\begin{split}
\mathbb{E}[\sup_{t\leq T} |X(\tilde{t})-X_{n}(\tilde{t})|^{2}] \leq
4(1+2C_{1}^{2})\mathbb{E}\int_{0}^{T} ||\sigma (X_{n}(\tilde{s}))-\sigma (X(\tilde{s}))||^{2}ds\\
+4(1+2C_{1}^{2})\mathbb{E}\int_{0}^{T} ||\sigma (X_{n}(\phi_{n}(\tilde{s})))-\sigma (X_{n}(\tilde{s}))||^{2}ds\\
+2\mathbb{E}\int_{0}^{T} |b(X_{n}(\phi_{n}(\tilde{s})))-b(X_{n}(\tilde{s}))|^{2}ds
+2\mathbb{E}\int_{0}^{T} |X_{n}(\tilde{s})-X(\tilde{s})|^{2}ds\\
+4\mathbb{E}\int_{0}^{T} |b(X_{n}(\tilde{s}))-b(X(\tilde{s}))||X_{n}(\tilde{s})-X(\tilde{s})|ds.
\end{split}
\end{equation*}

Thanks to assumption \eqref{H1}, we get
\begin{equation*}
\begin{split}
\mathbb{E}[\sup_{t\leq T} |X(\tilde{t})-X_{n}(\tilde{t})|^{2}] \leq
C\log{N(\gamma)}\mathbb{E}\int_{0}^{T} |X_{n}(\tilde{s})-X(\tilde{s})|^{2}ds
+ CT\frac{\log{N(\gamma)}}{N(\gamma)^{\mu}}\\
+C\log{N(\gamma)}\mathbb{E}\int_{0}^{T} |X_{n}(\phi_{n}(\tilde{s})))-X_{n}(\tilde{s})|^{2}ds
+CT\frac{\log{N(\gamma)}}{N(\gamma)^{\mu}}\\
+C(\log{N(\gamma)})^{2}\mathbb{E}\int_{0}^{T} |X_{n}(\phi_{n}(\tilde{s}))-X_{n}(\tilde{s})|^{2}ds
+ CT(\frac{\log{N(\gamma)}}{N(\gamma)^{\mu}})^{2}\\
+2\mathbb{E}\int_{0}^{T} |X_{n}(\tilde{s})-X(\tilde{s})|^{2}ds\\
+C\log{N(\gamma)}\mathbb{E}\int_{0}^{T} |X_{n}(\tilde{s})-X(\tilde{s})|^{2}ds
+CT\frac{\log{N(\gamma)}}{N(\gamma)^{\mu}}.
\end{split}
\end{equation*}

In  view of the definition of $\tau_{n}$, it follows that
\begin{equation*}
\begin{split}
\mathbb{E}[\sup_{t\leq T} |X(\tilde{t})-X_{n}(\tilde{t})|^{2}]\leq
C\log{N(\gamma)}\int_{0}^{T} \mathbb{E}[\sup_{u\leq s}|X(\tilde{u})-X_{n}(\tilde{u})|^{2}]ds\\
+CTa^{-2n}[(\log{N(\gamma)})^{2}+\log{N(\gamma)}]+CT\frac{\log{N(\gamma)}}{N(\gamma)^{\mu}}.
\end{split}
\end{equation*}

Using Gronwall's lemma, we obtain
\begin{equation} \label{joie}
\mathbb{E}[\sup_{t\leq T} |X(\tilde{t})-X_{n}(\tilde{t})|^{2}]
\leq CT\{a^{-2n}[(\log{N(\gamma)})^{2}+\log{N(\gamma)}]+\frac{\log{N(\gamma)}}{N(\gamma)^{\mu}}\}N(\gamma)^{CT}.
\end{equation}

Letting $n$ tends to $+\infty$ in \eqref{joie}, we get:
\begin{equation} \label{joie1}
 \limsup_{n\rightarrow +\infty}\mathbb{E}[\sup_{t\leq T} |X(\tilde{t})-X_{n}(\tilde{t})|^{2}]
 \leq CT\frac{\log{N(\gamma)}}{N(\gamma)^{\mu-CT}}.
\end{equation}

Using \eqref{continuity} [with $N=N(\gamma)$], \eqref{continuity2}, \eqref{continuity3} and \eqref{joie1},
it follows that
\begin{equation}\label{continuity5}
 \limsup_{n\rightarrow +\infty} \mathbb{E}[\sup_{t \leq T} {|X(t\wedge\tau_{n})-X_{n}(t\wedge\tau_{n})|}^{2}]
 \leq 60TM^{2}\gamma + CT\frac{\log{N(\gamma)}}{N(\gamma)^{\mu-CT}}.
\end{equation}

For $T<\mu/C$, we tend $\gamma$ to 0 in \eqref{continuity5} to obtain
\begin{equation}\label{continuity66}
 \limsup_{n\rightarrow +\infty} \mathbb{E}[\sup_{t \leq T} {|X(t\wedge\tau_{n})-X_{n}(t\wedge\tau_{n})|}^{2}]=0.
\end{equation}

This implies that
\begin{equation}\label{continuity6}
 \lim_{n\rightarrow +\infty} \mathbb{E}[\sup_{t \leq T} {|X(t\wedge\tau_{n})-X_{n}(t\wedge\tau_{n})|}^{2}]=0.
\end{equation}

Starting again from $\mu/C$ and applying the same arguments as in the first part of the proof, we get for any $T\in [\mu/C;2\mu/C[$
\begin{equation*} \label{joie4}
\lim_{n\rightarrow +\infty}\mathbb{E}[\sup_{t\leq T} |X(t\wedge\tau_{n})-X_{n}(t\wedge\tau_{n})|^{2}]=0.
\end{equation*}

For $k\in\mathbb{N}$, we set $T_{k}:=k\mu/C$.
Then, starting from $T_{k}$ and using the same arguments as in the first part of the proof, we obtain
for any $T\in [T_{k};T_{k+1}[$
\begin{equation*} \label{joie5}
\lim_{n\rightarrow +\infty}\mathbb{E}[\sup_{t\leq T} |X(t\wedge\tau_{n})-X_{n}(t\wedge\tau_{n})|^{2}]=0.
\end{equation*}

Since for any $T>0$, there exists a unique $k_{0}$ such that $T\in [T_{k_{0}};T_{k_{0}+1}[$ and
\begin{equation} \label{joie6}
\sum_{k=0}^{k_{0}}
\lim_{n\rightarrow +\infty}\mathbb{E}[\sup_{t\in [0;T]\cap[T_{k};T_{k+1}[} |X(t\wedge\tau_{n})-X_{n}(t\wedge\tau_{n})|^{2}]=0,
\end{equation}
it follows that for any $T>0$
\begin{equation}\label{convergenceps}
 \lim_{n\rightarrow +\infty}\mathbb{E}[\sup_{t\leq T} |X(t\wedge\tau_{n})-X_{n}(t\wedge\tau_{n})|^{2}]=0.
\end{equation}

Notice that
\begin{equation}\label{convergenceps1}
\begin{split}
 \mathbb{E}\sup_{t \leq T} {|X(t)-X_{n}(t)|}^{2}
 \leq 3\mathbb{E}\sup_{t \leq T} {|X(t)-X(t\wedge\tau_{n})|}^{2}\\
 +3\mathbb{E}\sup_{t \leq T} {|X(t\wedge\tau_{n})-X_{n}(t\wedge\tau_{n})|}^{2} \\
 +3\mathbb{E}\sup_{t \leq T} {|X_{n}(t\wedge\tau_{n})-X_{n}(t)|}^{2}.
 \end{split}
\end{equation}

Furthermore, we have
\begin{equation}\label{convergenceps2}
 \begin{split}
\mathbb{E}\sup_{t \leq T} {|X(t)-X(t\wedge\tau_{n})|}^{2} \leq
2\mathbb{E}\sup_{t\leq T}|\int_{t\wedge\tau_{n}}^{t} b(X(s))ds|^{2}\\
+2\mathbb{E}\sup_{t\leq T}|\int_{t\wedge\tau_{n}}^{t} \sigma (X(s))dW(s)|^{2}.
 \end{split}
\end{equation}

Then, using Doob's inequality, we obtain
\begin{equation}\label{convergenceps3}
 \mathbb{E}\sup_{t \leq T} {|X(t)-X(t\wedge\tau_{n})|}^{2} \leq 10M^{2}T\mathbb{P}[\tau_{n}\leq T]
\end{equation}
where $M$ is the uniform bound on $\sigma$ and $b$.

In a same manner, one may easily obtained the following:
\begin{equation}\label{convergenceps4}
 \mathbb{E}\sup_{t \leq T} {|X_{n}(t)-X_{n}(t\wedge\tau_{n})|}^{2} \leq 10M^{2}T\mathbb{P}[\tau_{n}\leq T].
\end{equation}

In view of \eqref{convergenceps1}, \eqref{convergenceps3} and \eqref{convergenceps4}, it follows that
\begin{equation}\label{convergenceps5}
\begin{split}
 \mathbb{E}\sup_{t \leq T} {|X(t)-X_{n}(t)|}^{2} &\leq 60M^{2}T\mathbb{P}[\tau_{n}\leq T] \\
 &+3\mathbb{E}\sup_{t \leq T} {|X(t\wedge\tau_{n})-X_{n}(t\wedge\tau_{n})|}^{2}.
 \end{split}
\end{equation}

Using \eqref{estimates123} and \eqref{convergenceps},
and letting $n$ tends to $+\infty$ in \eqref{convergenceps5}, we get:
\begin{equation}
\lim_{n\rightarrow +\infty}\mathbb{E}\sup_{t\leq T} |X(t)-X_{n}(t)|^{2}=0.
\end{equation}

The proof is completed.
\end{proof}

\subsection{The first key lemma}

Let $\varepsilon>0$ and consider the following SDE:
\begin{equation}
 X^{\varepsilon}(t)=x+\int_{0}^{t}b(X^{\varepsilon}(s))ds+\sqrt{\varepsilon}\int_{0}^{t}\sigma(X^{\varepsilon}(s))dW(s)
\end{equation}
with its associated Euler approximation
\begin{equation}
X^{\varepsilon}_{n}(t)=x+\int_{0}^{t}b(X^{\varepsilon}_{n}(\phi_{n}(s)))ds
+\sqrt{\varepsilon}\int_{0}^{t}\sigma(X^{\varepsilon}_{n}(\phi_{n}(s)))dW(s)
\end{equation}
where $\phi_{n}(s)$ is defined as in the proof of Theorem \ref{approx}.

\begin{lemma} \label{expo}
Under the hypothesis of Theorem \ref{approx}, we have for any $\delta>0$,
\begin{equation}\label{7.1}
 \lim_{n\rightarrow\infty} \limsup_{\varepsilon\rightarrow 0} \varepsilon \log
\mathbb{P} ( \sup_{0\leq t \leq T} | X^{\varepsilon}(t)-X^{\varepsilon}_{n}(t)|\geq\delta )=-\infty.
\end{equation}
\end{lemma}

\begin{proof}{}
We will  proceed as in Deuschel and Stroock \cite{Deuschel-Stroock}.
For $\rho>0$, we define
\begin{equation*}
 \tau_{n,\varepsilon}^{\rho}:=\inf\{t>0; |X^{\varepsilon}_{n}(t)-X^{\varepsilon}_{n}(\phi_{n}(t))|\geq\rho\}
\end{equation*}
and
\begin{equation*}
 \varsigma^{\rho}_{n,\varepsilon}:=
 \inf\{t>0; |X^{\varepsilon}_{n}(t)-X^{\varepsilon}(t)|\geq\delta\}\wedge\tau_{n,\varepsilon}^{\rho}.
\end{equation*}
Clearly,
\begin{equation*}
 \mathbb{P}( \sup_{0\leq t\leq T}| X^{\varepsilon}(t)-X^{\varepsilon}_{n}(t)|\geq\delta )
 \leq \mathbb{P}(\tau_{n,\varepsilon}^{\rho}\leq T) + \mathbb{P}(\varsigma^{\rho}_{n,\varepsilon}\leq T).
\end{equation*}
It then suffices to prove that for each $\rho>0$
\begin{equation} \label{1.x}
 \lim_{n\rightarrow +\infty} \limsup_{\varepsilon \rightarrow 0} \varepsilon
 \log{\mathbb{P}(\tau_{n,\varepsilon}^{\rho}\leq T)}=-\infty
\end{equation}
 which implies that
\begin{equation} \label{1.y}
 \lim_{\rho\rightarrow 0} \limsup_{n\rightarrow +\infty}
 \limsup_{\varepsilon\rightarrow 0} \varepsilon \log{\mathbb{P}(\varsigma^{\rho}_{n,\varepsilon}\leq T)}=-\infty.
\end{equation}
To prove $\eqref{1.x}$, we replace $A^{2}$ by $\varepsilon A^{2}$, $c^{n}=2^{n}/a^{2n}$ by $2^{n}\rho^{2}$
in the estimate \eqref{estimates123}
to get
\begin{equation*}
 \mathbb{P} (\tau_{n,\varepsilon}^{\rho}\leq T) \leq 2^{n+1}d([T]+1)exp\{ -2^{n}/4\varepsilon dA^{2}\}
\end{equation*}
from which \eqref{1.x}  easily follows.

We shall prove \eqref{1.y}. For $N\in\mathbb{N}$, we set
$$\zeta_{N}^{n,\varepsilon}:=\inf\{t>0; |X^{\varepsilon}_{n}(t)|>N \mbox{ or } |X^{\varepsilon}(t)|>N \}.$$
For $y\in\mathbb{R}^{d}$, we define a function $f$ by
$$f(y):=(\rho^{2}+|y|^{2}+\frac{1}{N^{\mu}})^{1/\varepsilon}.$$

It is not difficult to show that there exists a positive constant $C<+\infty$ such that the gradient $D^{1}f$ and the Hessian
matrix $D^{2}f$ of $f$ satisfy:
$$|D^{1}f(y)|\leq\frac{C}{\varepsilon}(\rho^{2}+|y|^{2}+\frac{1}{N^{\mu}})^{-1/2}f(y)$$
and
$$|D^{2}f(y)|\leq\frac{C}{\varepsilon^{2}}(\rho^{2}+|y|^{2}+\frac{1}
{N^{\mu}})^{-1}f(y).$$
Set $Y_{n,\varepsilon}(t):=X_{n}^{\varepsilon}(t)-X^{\varepsilon}(t)$. We use It\^o's formula to get
\begin{equation*}
\begin{split}
 f(Y_{n,\varepsilon}(t))=f(Y_{n,\varepsilon}(0))+\int_{0}^{t} \langle D^{1}f(Y_{n,\varepsilon}(s)),
 [ b(X_{n}^{\varepsilon}(\phi_{n}(s)))-b(X^{\varepsilon}(s))] \rangle ds \\
 +\int_{0}^{t} \sqrt{\varepsilon}\langle D^{1}f(Y_{n,\varepsilon}(s)),
 [ \sigma(X_{n}^{\varepsilon}(\phi_{n}(s)))-\sigma(X^{\varepsilon}(s))]dW(s)\rangle \\
 +\frac{\varepsilon}{2}\int_{0}^{t} \mbox{Trace}\{ D^{2}f(Y_{n,\varepsilon}(s))
 [\sigma(X_{n}^{\varepsilon}(\phi_{n}(s)))-\sigma(X^{\varepsilon}(s))] \\
 [\sigma(X_{n}^{\varepsilon}(\phi_{n}(s)))-\sigma(X^{\varepsilon}(s))]^{\top}\}ds.
 \end{split}
\end{equation*}
Thanks to Burkholder's inequality, we get for any $T>0$
\begin{equation*}
\begin{split}
 \mathbb{E} \sup_{t\leq T} f(Y_{n,\varepsilon}(t))\leq f(Y_{n,\varepsilon}(0))
 + \mathbb{E}\int_{0}^{T} |D^{1}f(Y_{n,\varepsilon}(s))|
 |b(X_{n}^{\varepsilon}(\phi_{n}(s)))-b(X^{\varepsilon}(s))|ds \\
 +C_{1}\mathbb{E}{\left( \int_{0}^{T} \varepsilon |D^{1}f(Y_{n,\varepsilon}(s))|^{2}
 ||\sigma(X_{n}^{\varepsilon}(\phi_{n}(s)))-\sigma(X^{\varepsilon}(s))||^{2}ds \right)}^{1/2}\\
 +\frac{\varepsilon}{2}\mathbb{E}\int_{0}^{T} |D^{2}f(Y_{n,\varepsilon}(s))|
 {||\sigma(X_{n}^{\varepsilon}(\phi_{n}(s)))-\sigma(X^{\varepsilon}(s))||}^{2}ds.
 \end{split}
\end{equation*}
This implies that
\begin{equation*}
\begin{split}
 \mathbb{E} \sup_{t\leq T} &f(Y_{n,\varepsilon}(t)) \leq f(Y_{n,\varepsilon}(0)) \\
 &+\frac{C}{\varepsilon}\mathbb{E}\int_{0}^{T}
 (\rho^{2}+|Y_{n,\varepsilon}(s)|^{2}+\frac{1}{N^{\mu}})^{-1/2}
 |b(X_{n}^{\varepsilon}(\phi_{n}(s)))-b(X^{\varepsilon}(s))||f(Y_{n,\varepsilon}(s))|ds\\
 &+\mathbb{E}{\left( \int_{0}^{T} \frac{C}{\varepsilon} (\rho^{2}+|Y_{n,\varepsilon}(s)|^{2}+\frac{1}{N^{\mu}})^{-1}
 ||\sigma(X_{n}^{\varepsilon}(\phi_{n}(s)))-\sigma(X^{\varepsilon}(s))||^{2}|f(Y_{n,\varepsilon}(s))|^{2}ds \right)}^{1/2}\\
 &+\frac{C}{2\varepsilon}\mathbb{E}\int_{0}^{T} (\rho^{2}+|Y_{n,\varepsilon}(s)|^{2}
 +\frac{1}{N^{\mu}})^{-1}|f(Y_{n,\varepsilon}(s))|
 {||\sigma(X_{n}^{\varepsilon}(\phi_{n}(s)))-\sigma(X^{\varepsilon}(s))||}^{2}ds.
 \end{split}
\end{equation*}
It follows that
\begin{equation*}
\begin{split}
 \mathbb{E} \sup_{t\leq T} &f(Y_{n,\varepsilon}(t)) \leq f(Y_{n,\varepsilon}(0))
 +\frac{1}{2}\mathbb{E} \sup_{t\leq T} f(Y_{n,\varepsilon}(t))\\
 &+\frac{C}{\varepsilon}\mathbb{E}\int_{0}^{T}
 (\rho^{2}+|Y_{n,\varepsilon}(s)|^{2}+\frac{1}{N^{\mu}})^{-1/2}
 |b(X_{n}^{\varepsilon}(\phi_{n}(s)))-b(X^{\varepsilon}(s))||f(Y_{n,\varepsilon}(s))|ds \\
 &+\frac{C}{\varepsilon}\mathbb{E}\int_{0}^{T} (\rho^{2}+|Y_{n,\varepsilon}(s)|^{2}+\frac{1}{N^{\mu}})^{-1}
 ||\sigma(X_{n}^{\varepsilon}(\phi_{n}(s)))-\sigma(X^{\varepsilon}(s))||^{2}|f(Y_{n,\varepsilon}(s))|ds.
\end{split}
\end{equation*}
Thus
\begin{equation*}
\begin{split}
\mathbb{E} \sup_{t\leq T} &f(Y_{n,\varepsilon}(t)) \leq 2f(Y_{n,\varepsilon}(0)) \\
&+\frac{C}{\varepsilon}\mathbb{E}\int_{0}^{T}
(\rho^{2}+|Y_{n,\varepsilon}(s)|^{2}+\frac{1}{N^{\mu}})^{-1/2}|b(X_{n}^{\varepsilon}
(\phi_{n}(s)))-b(X^{\varepsilon}(s))||f(Y_{n,\varepsilon}(s))|ds \\
&+\frac{C}{\varepsilon}\mathbb{E}\int_{0}^{T} (\rho^{2}+|Y_{n,\varepsilon}(s)|^{2}+\frac{1}{N^{\mu}})^{-1}
||\sigma(X_{n}^{\varepsilon}(\phi_{n}(s)))-\sigma(X^{\varepsilon}(s))||^{2}|f(Y_{n,\varepsilon}(s))|ds.
 \end{split}
\end{equation*}
Using \eqref{H1} and triangular inequality, we get
\begin{equation*}
\begin{split}
 \mathbb{E} \sup_{t\leq T} &f(Y_{n,\varepsilon}(t\wedge\zeta_{N}^{n,\varepsilon})) \leq 2f(Y_{n,\varepsilon}(0)) \\
 &+\frac{C\log{N}}{\varepsilon}\mathbb{E}\int_{0}^{T}
 (\rho^{2}+|Y_{n,\varepsilon}(s\wedge\zeta_{N}^{n,\varepsilon})|^{2}
 +\frac{1}{N^{\mu}})^{-1/2}f(Y_{n,\varepsilon}(s\wedge\zeta_{N}^{n,\varepsilon}))\\
 &\left(|X_{n}^{\varepsilon}(\phi_{n}
 (s\wedge\zeta_{N}^{n,\varepsilon}))-X_{n}^{\varepsilon}(s\wedge\zeta_{N}^{n,\varepsilon})|+
 |Y_{n,\varepsilon}(s\wedge\zeta_{N}^{n,\varepsilon})|+\frac{1}{N^{\mu}}\right)ds \\
 &+\frac{C\log{N}}{\varepsilon}\mathbb{E}\int_{0}^{T}
 (\rho^{2}+|Y_{n,\varepsilon}(s\wedge\zeta_{N}^{n,\varepsilon})|^{2}+\frac{1}{N^{\mu}})^{-1}
 f(Y_{n,\varepsilon}(s\wedge\zeta_{N}^{n,\varepsilon}))\\
 &\left(|X_{n}^{\varepsilon}(\phi_{n}(s\wedge\zeta_{N}^{n,\varepsilon}))-X_{n}^{\varepsilon}
 (s\wedge\zeta_{N}^{n,\varepsilon})|^{2}
 +|Y_{n,\varepsilon}(s\wedge\zeta_{N}^{n,\varepsilon})|^{2}+\frac{1}{N^{\mu}}\right)ds.
 \end{split}
\end{equation*}
In view of the definition of  $\varsigma^{\rho}_{n,\varepsilon}$, it follows that
\begin{equation*}
 \mathbb{E} \sup_{t\leq T} f(Y_{n,\varepsilon}(t\wedge\varsigma^{\rho}_{n,\varepsilon}\wedge\zeta_{N}^{n,\varepsilon}))
 \leq 2f(Y_{n,\varepsilon}(0))
 + \frac{C\log{N}}{\varepsilon}\mathbb{E}\int_{0}^{T}
 f(Y_{n,\varepsilon}(s\wedge\varsigma^{\rho}_{n,\varepsilon}\wedge\zeta_{N}^{n,\varepsilon}))ds.
\end{equation*}
This implies that
\begin{equation*}
 \mathbb{E}[\sup_{t\leq T} f(Y_{n,\varepsilon}(t\wedge\varsigma^{\rho}_{n,\varepsilon}\wedge\zeta_{N}^{n,\varepsilon}))]
 \leq 2f(Y_{n,\varepsilon}(0)) + \frac{C\log{N}}{\varepsilon}\int_{0}^{T}
 \mathbb{E}[\sup_{u\leq s}f(Y_{n,\varepsilon}(u\wedge\varsigma^{\rho}_{n,\varepsilon}\wedge\zeta_{N}^{n,\varepsilon}))]ds.
\end{equation*}
Thanks to Gronwall lemma, it follows that
\begin{equation*}
 \mathbb{E} \sup_{t\leq T} f(Y_{n,\varepsilon}(t\wedge\varsigma^{\rho}_{n,\varepsilon}\wedge\zeta_{N}^{n,\varepsilon}))\leq
 2(\rho^{2}+\frac{1}{N^{\mu}})^{1/\varepsilon} N^{CT/\varepsilon}.
\end{equation*}
We deduce that
\begin{align*}
 \mathbb{E}\sup_{t\leq T} {\left( \rho^{2}+
 |Y_{n,\varepsilon}(t\wedge\varsigma^{\rho}_{n,\varepsilon}\wedge\zeta_{N}^{n,\varepsilon}|^{2} \right)}^{1/\varepsilon}
 & \leq \mathbb{E} \sup_{t\leq T} f(Y_{n,\varepsilon}(t\wedge\varsigma^{\rho}_{n,\varepsilon}
 \wedge\zeta_{N}^{n,\varepsilon})) \\
 & \leq 2(\rho^{2}+\frac{1}{N^{\mu}})^{1/\varepsilon} N^{CT/\varepsilon}.
\end{align*}
Since
\begin{equation*}
 \mathbb{P}(\varsigma^{\rho}_{n,\varepsilon}\leq T;\zeta_{N}^{n,\varepsilon}>T)\leq
 (\rho^{2}+\delta^{2})^{-1/\varepsilon}\mathbb{E}\sup_{t\leq T} {\left( \rho^{2}+
 |Y_{n,\varepsilon}(t\wedge\varsigma^{\rho}_{n,\varepsilon}\wedge\zeta_{N}^{n,\varepsilon}|^{2} \right)}^{1/\varepsilon},
\end{equation*}
it follows that
\begin{equation*}
 \mathbb{P}(\varsigma^{\rho}_{n,\varepsilon}\leq T; \zeta_{N}^{n,\varepsilon}>T)\leq
 2(\rho^{2}+\delta^{2})^{-1/\varepsilon}(\rho^{2}N^{CT}+N^{CT-\mu})^{1/\varepsilon}.
\end{equation*}
Therefore, since
\begin{equation}\label{bestineq}
 \mathbb{P}(\varsigma^{\rho}_{n,\varepsilon}\leq T)\leq  \mathbb{P}(\zeta^{n,\varepsilon}_{N}\leq T) +
 \mathbb{P}(\varsigma^{\rho}_{n,\varepsilon}\leq T; \zeta_{N}^{n,\varepsilon}>T),
\end{equation}
it follows that
\begin{equation}
  \mathbb{P}(\varsigma^{\rho}_{n,\varepsilon}\leq T)\leq  \mathbb{P}(\zeta^{n,\varepsilon}_{N}\leq T) +
  2(\rho^{2}+\delta^{2})^{-1/\varepsilon}(\rho^{2}N^{CT}+N^{CT-\mu})^{1/\varepsilon}.
\end{equation}
Hence,
\begin{equation}
 \limsup_{\varepsilon\rightarrow 0}\varepsilon\log{\mathbb{P}(\varsigma^{\rho}_{n,\varepsilon}\leq T)} \leq
 \{\limsup_{\varepsilon\rightarrow 0}\varepsilon\log \mathbb{P} (\zeta^{n,\varepsilon}_{N}\leq T)\}
 \vee \{\log \frac{\rho^{2}N^{CT}+N^{CT-\mu}}{\rho^{2}+\delta^{2}} \}.
\end{equation}
Taking the supremum on $n$ and passing the limit on $\rho$, it follows
\begin{equation}\label{doux}
\begin{split}
 \lim_{\rho\rightarrow 0}\sup_{n\geq0}\limsup_{\varepsilon\rightarrow 0}\varepsilon
 \log{\mathbb{P}(\varsigma^{\rho}_{n,\varepsilon}\leq T)} &\leq
 \{\sup_{n\geq0}\limsup_{\varepsilon\rightarrow 0}\varepsilon\log{\mathbb{P} (\zeta^{n,\varepsilon}_{N}\leq T)}\} \\
 &\vee \{(CT-\mu)\log{N}-2\log{\delta} \}.
\end{split}
\end{equation}
Since $\sigma$ and $b$ are bounded, one can easily prove thanks to Lemma \ref{estimate} that for any $T>0$,
\begin{equation}
 \lim_{N\rightarrow+\infty} \sup_{n\geq0} \limsup_{\varepsilon\rightarrow0}\varepsilon
 \log{\mathbb{P}[\zeta^{n,\varepsilon}_{N}\leq T]}=-\infty.
\end{equation}
For any $T<\mu/C$, we tend $N$ to $+\infty$ in \eqref{doux} to get:
\begin{equation}
 \lim_{\rho\rightarrow 0}\sup_{n\geq0}\limsup_{\varepsilon\rightarrow 0}
 \varepsilon\log{\mathbb{P}(\varsigma^{\rho}_{n,\varepsilon}\leq T)}=-\infty.
\end{equation}
Starting again from $\mu/C$ and using the same arguments as in the first part of the proof,
we obtain for $T\in [\mu/C;2\mu/C[$
\begin{equation*}
 \lim_{\rho\rightarrow 0}\sup_{n\geq0}\limsup_{\varepsilon\rightarrow 0}
 \varepsilon\log{\mathbb{P}(\varsigma^{\rho}_{n,\varepsilon}\leq T)}=-\infty.
\end{equation*}
For $k\in\mathbb{N}$, we set $T_{k}:=k\mu/C$. Then, starting again from $k\mu/C$ and using
the same arguments as above, it follows that
for any $T\in [T_{k};T_{k+1}[$
\begin{equation*}
 \lim_{\rho\rightarrow 0}\sup_{n\geq0}\limsup_{\varepsilon\rightarrow 0}
 \varepsilon\log{\mathbb{P}(\varsigma^{\rho}_{n,\varepsilon}\leq T)}=-\infty.
\end{equation*}
Since for any $T>0$, there exists a unique $k_{0}$ such that $T\in [T_{k_{0}};T_{k_{0}+1}[$, it follows that
\begin{equation*}
\lim_{\rho\rightarrow 0}\sup_{n\geq0}\limsup_{\varepsilon\rightarrow 0}
 \varepsilon\log{\mathbb{P}(\varsigma^{\rho}_{n,\varepsilon}\leq T)} \leq
\vee_{k=0}^{k_{0}}\left( \lim_{\rho\rightarrow 0}\sup_{n\geq0}\limsup_{\varepsilon\rightarrow 0}
 \varepsilon\log{\mathbb{P}(\varsigma^{\rho}_{n,\varepsilon} \leq T_{k}}) \right)
\end{equation*}
from which \eqref{1.y} easily follows.
The proof is now finished.
\end{proof}

\subsection{The second key lemma}

For $x\in\mathbb{R}^{m}$, we denote by $C_{x}([0,T],\mathbb{R}^{m})$ the space of continuous functions
from $[0,T]$ into $\mathbb{R}^{m}$
with initial value $x$. For $g\in C_{0}([0,T],\mathbb{R}^{m})$, we define
\begin{equation}
 e(g)=\left\{
\begin{array}{cc}
\int_{0}^{T} {| \dot{g}(t)|}^{2}dt \mbox{ if g is absolutely continuous }\\
+\infty \mbox{ otherwise. }
\end{array}
 \right.
\end{equation}
For an absolutely continuous function $h\in C_{0}([0,T],\mathbb{R}^{m})$, we consider the following
ordinary differential equation (ODE in short) on
$\mathbb{R}^{d}$
\begin{equation} \label{ODE}
dX_{h}(t)=\left( \sigma(X_{h}(t))\dot{h}(t)+b(X_{h}(t)) \right) dt,\, X_{h}(0)=x\in\mathbb{R}^{d}.
\end{equation}
 Under assumption \eqref{H1} and the boundedness of the coefficients $\sigma$ and $b$, the existence and
uniqueness of  solution holds for the ODE \eqref{ODE}.

Let us consider the following Euler approximation of the ODE \eqref{ODE}
\begin{equation}
dX^{n}_{h}(t)=\left( \sigma(X^{n}_{h}(\phi_{n}(t)))\dot{h}(t)
+b(X^{n}_{h}(\phi_{n}(t))) \right) dt, \ \ \  X^{n}_{h}(0)=x\in\mathbb{R}^{d}.
\end{equation}

\begin{lemma}\label{expo2}

Let $h\in C_{0}([0,T],\mathbb{R}^{m})$ such that $e(h)<+\infty$.
Then, for any $\alpha>0$, we have
\begin{equation} \label{odeapprox}
\lim_{n\rightarrow +\infty} \sup_{\{h;e(h)\leq\alpha\}}(\sup_{0\leq t\leq T}| X^{n}_{h}(t)-X_{h}(t)|)=0.
\end{equation}

\end{lemma}

\begin{proof}{}
For $N\in\mathbb{N}$, let $\zeta_{N}^{n, h}:=\inf\{t>0: |X_{h}(t)|>N \, or \, | X^{n}_{h}(t)| >N\}$.
Since $\sigma$ and b are bounded, then \ $\lim_{N\rightarrow\infty}\zeta_{N}^{n, h} = \infty$ uniformly with
respect to $n$ and $h$.
For $t\in [k2^{-n};(k+1)2^{-n}[$, we have
\begin{equation*}
 X^{n}_{h}(t)-X^{n}_{h}(k2^{-n})=\sigma(X^{n}_{h}(k2^{-n}))(h(t)-h(k2^{-n}))+b(X^{n}_{h}(k2^{-n}))(t-k2^{-n}).
\end{equation*}
Since $\sigma$ and b are bounded and \ $|h(t)-h(k2^{-n})|\leq 2^{-n/2}\sqrt{e(h)}$,
then the following estimate holds
\begin{equation}\label{estimate12}
 |X^{n}_{h}(t)-X^{n}_{h}(\phi_{n}(t))|\leq 2^{-n/2}M[\sqrt{e(h)}+1], \, \mbox{ for any } t>0,
\end{equation}
where $M$ is the uniform bound on $\sigma$ and $b$.

\noindent Furthermore,  we have for any $t>0$
\begin{equation*}
\begin{split}
X^{n}_{h}(t)-X_{h}(t) = \int_{0}^{t}
[ \sigma(X^{n}_{h}(\phi_{n}(s)))-\sigma(X_{h}(s)) ] \dot{h}(s)ds \\
+ \int_{0}^{t} [ b(X^{n}_{h}(\phi_{n}(s)))-b(X_{h}(s)) ] ds.
\end{split}
\end{equation*}
This implies that
\begin{equation*}
\begin{split}
|X^{n}_{h}(t)-X_{h}(t)| \leq \int_{0}^{t}
||\sigma(X^{n}_{h}(\phi_{n}(s)))-\sigma(X_{h}(s))|| |\dot{h}(s)|ds \\
+ \int_{0}^{t}|b(X^{n}_{h}(\phi_{n}(s)))-b(X_{h}(s))|ds.
\end{split}
\end{equation*}
Using triangular inequality, it follows that
\begin{equation*}
\begin{split}
|X^{n}_{h}(t)&-X_{h}(t)| \leq \\
&\int_{0}^{t}
||\sigma(X^{n}_{h}(\phi_{n}(s)))-\sigma(X^{n}_{h}(s))|| |\dot{h}(s)|ds
+ \int_{0}^{t}|b(X^{n}_{h}(s))-b(X_{h}(s))|ds\\
&+\int_{0}^{t} ||\sigma(X^{n}_{h}(s))-\sigma(X_{h}(s))|||\dot{h}(s)|ds
+ \int_{0}^{t}|b(X^{n}_{h}(\phi_{n}(s)))-b(X^{n}_{h}(s))|ds.
\end{split}
\end{equation*}
Thanks to conditions \eqref{H1}, it follows that
\begin{equation*}
\begin{split}
|X^{n}_{h}(t\wedge\zeta_{N}^{n, h})-X_{h}(t\wedge\zeta_{N}^{n, h})| \leq C\sqrt{\log{N}}\int_{0}^{t\wedge\zeta_{N}^{n, h}}
|X^{n}_{h}(\phi_{n}(s))-X^{n}_{h}(s)| |\dot{h}(s)|ds \\
+C\sqrt{\log{N}}\int_{0}^{t\wedge\zeta_{N}^{n, h}}|X^{n}_{h}(s)-X_{h}(s)| |\dot{h}(s)|ds\\
+ C\log{N}\int_{0}^{t\wedge\zeta_{N}^{n, h}}|X^{n}_{h}(\phi_{n}(s))-X^{n}_{h}(s)|ds \\
+ C\log{N}\int_{0}^{t\wedge\zeta_{N}^{n, h}}|X^{n}_{h}(s)-X_{h}(s)|ds \\
+ 2Ct\frac{\log{N}}{N^{\mu}}+ 2C\frac{\log{N}}{N^{\mu}}\int_{0}^{t} |\dot{h}(s)|ds.
\end{split}
\end{equation*}
Using Cauchy-Schwartz's inequality and the estimates \eqref{estimate12}, we get
\begin{equation*}
\begin{split}
|X^{n}_{h}(t\wedge\zeta_{N}^{n, h})-X_{h}(t\wedge\zeta_{N}^{n, h})| \leq MC\log{N}(\sqrt{e(h)}+1)t\sqrt{e(h)}2^{-n/2} \\
+C\log{N}\int_{0}^{t}|X^{n}_{h}(s\wedge\zeta_{N}^{n, h})-X_{h}
(s\wedge\zeta_{N}^{n, h})|(1+|\dot{h}(s\wedge\zeta_{N}^{n, h})|)ds \\
+ C\frac{\log{N}}{N^{\mu}}(\sqrt{e(h)}+1)t.
\end{split}
\end{equation*}
Thanks to Gronwall lemma and Cauchy-Schwartz' inequality, it follows that
\begin{equation*}
| X^{n}_{h}(t\wedge\zeta_{N}^{n, h})-X_{h}(t\wedge\zeta_{N}^{n, h})|\leq Ct(\sqrt{e(h)}+1)\log{N}
\left( 2^{-n/2}\sqrt{e(h)}+\frac{1}{N^{\mu}} \right) N^{Ct(\sqrt{e(h)}+1)}.
\end{equation*}
We take the supremum on $t$ and $h$ then we tend $n$ to $+\infty$ in the previous inequality, to get
\begin{equation}\label{23}
\limsup_{n\rightarrow +\infty}\sup_{\{h;e(h)\leq\alpha\}}
(\sup_{0\leq t\leq T}| X^{n}_{h}(t\wedge\zeta_{N}^{n, h})-X_{h}(t\wedge\zeta_{N}^{n, h})|)\leq
\frac{C_{\alpha}T\log N}{N^{\mu-C_{\alpha}T}}
\end{equation}
where $C_{\alpha}$ is a constant which only depends on a given positive real $\alpha$.

Let us notice that
\begin{equation}
\begin{split}
 |X_{h}^{n}(t)-X_{h}(t)| &\leq |X_{h}^{n}(t)-X_{h}^{n}(t\wedge\zeta_{N}^{n, h})|\\
 &+| X^{n}_{h}(t\wedge\zeta_{N}^{n, h})-X_{h}(t\wedge\zeta_{N}^{n, h})| \\
 &+ |X_{h}(t)-X_{h}(t\wedge\zeta_{N}^{n, h})|.
\end{split}
\end{equation}
Furthermore, we have
\begin{equation}
\begin{split}
|X_{h}^{n}(t)-X_{h}^{n}(t\wedge\zeta_{N}^{n, h})|\leq |\int_{t\wedge\zeta_{N}^{n, h}}^{t}
\sigma(X^{n}_{h}(\phi_{n}(s)))\dot{h}(s)ds|\\
+|\int_{t\wedge\zeta_{N}^{n, h}}^{t} b(X^{n}_{h}(\phi_{n}(s)))ds|.
\end{split}
\end{equation}
This implies thanks to Cauchy-Schwartz's inequality that,
\begin{equation}
 |X_{h}^{n}(t)-X_{h}^{n}(t\wedge\zeta_{N}^{n, h})|\leq M\sqrt{|t-t\wedge\zeta_{N}^{n, h}|}
 \sqrt{e(h)}+M|t-t\wedge\zeta_{N}^{n, h}|
\end{equation}
where $M$ is the uniform bound on $\sigma$ and $b$.
\\
In the same way, we obtain
\begin{equation}
 |X_{h}(t)-X_{h}(t\wedge\zeta_{N}^{n, h})|\leq M\sqrt{|t-t\wedge\zeta_{N}^{n, h}|}
 \sqrt{e(h)}+M|t-t\wedge\zeta_{N}^{n, h}|.
\end{equation}
Thus
\begin{equation}\label{grdpas}
\begin{split}
 |X_{h}^{n}(t)-X_{h}(t)|\leq 2M\sqrt{|t-t\wedge\zeta_{N}^{n, h}|}\sqrt{e(h)}+2M|t-t\wedge\zeta_{N}^{n, h}|\\
 +| X^{n}_{h}(t\wedge\zeta_{N}^{n, h})-X_{h}(t\wedge\zeta_{N}^{n, h})|.
\end{split}
\end{equation}
Taking the supremum on $t$ and $h$ in \eqref{grdpas} then letting $n$ tends to $+\infty$,
and using \eqref{23} we get
\begin{equation}\label{verslafin}
\begin{split}
\limsup_{n\rightarrow +\infty}\sup_{\{h;e(h)\leq\alpha\}}(\sup_{0\leq t\leq T}| X^{n}_{h}(t)&-X_{h}(t)| ) \leq
2M\sup_{n\geq 0}\sup_{\{h;e(h)\leq\alpha\}}(\sup_{0\leq t\leq T}|t-t\wedge\zeta_{N}^{n, h}| )\\
&+2M\sqrt{\alpha}\sup_{n\geq 0} \sup_{\{h;e(h)\leq\alpha\}}(\sup_{0\leq t\leq T}\sqrt{|t-t\wedge\zeta_{N}^{n, h}|} ) \\
&+\frac{C_{\alpha}T\log{N}}{N^{\mu-C_{\alpha}T}}.
\end{split}
\end{equation}
Note that, as $N$ goes to $+\infty$, $\zeta_{N}^{n, h}$ tends to $+\infty$ uniformly with respect to $n$ and $h$.  We fix a $\gamma>0$ and
 consider $N(\gamma)>Te^{\frac{1}{\gamma}}$ a natural number such that $\zeta^{n, h}_{N(\gamma)}>T$.
It follows from \eqref{verslafin} [with $N=N(\gamma)$] that
\begin{equation}\label{finproche}
\limsup_{n\rightarrow +\infty}\sup_{\{h;e(h)\leq\alpha\}}(\sup_{0\leq t\leq T}| X^{n}_{h}(t)-X_{h}(t)| )\leq
\frac{C_{\alpha}T\log{N(\gamma)}}{N(\gamma)^{\mu-C_{\alpha}T}}
\end{equation}
For $T<\mu/C_{\alpha}$, we tend $\gamma$ to $0$ in \eqref{finproche} and we get
\begin{equation}\label{24}
\limsup_{n\rightarrow +\infty}\sup_{\{h;e(h)\leq\alpha\}}(\sup_{0\leq t\leq T}| X^{n}_{h}(t)-X_{h}(t)|)=0,
\end{equation}
and this implies that, for any $T<\mu/C_{\alpha}$,
\begin{equation}\label{24154}
\lim_{n\rightarrow +\infty}\sup_{\{h;e(h)\leq\alpha\}}(\sup_{0\leq t\leq T}| X^{n}_{h}(t)-X_{h}(t)|)=0.
\end{equation}
Starting again from $\mu/C_{\alpha}$ and using the same arguments as above,
we show that for $T\in [\mu/C_{\alpha}; 2\mu/C_{\alpha}[$, we have
\begin{equation*}
 \lim_{n\rightarrow +\infty}\sup_{\{h;e(h)\leq\alpha\}}(\sup_{0\leq t\leq T}
| X^{n}_{h}(t)-X_{h}(t)|)=0.
\end{equation*}
It is clear that  the sequence $(T_{k}):=(k\mu/C_{\alpha})$ tends to $+\infty$ when $k$ goes to $+\infty$.
Hence, arguing as in the first part of this proof, we show that for any $T\in [T_{k};T_{k+1}[$,
\begin{equation*}
\lim_{n\rightarrow +\infty}\sup_{\{h;e(h)\leq\alpha\}}(\sup_{0\leq t\leq T}
| X^{n}_{h}(t)-X_{h}(t)|)=0.
\end{equation*}
Now, for any $T>0$, there exists a unique positive integer $k_{0}$ such that $T\in [T_{k_{0}},T_{k_{0}+1}[$ and we get
\begin{equation*}
\sum_{k=0}^{k_{0}}\lim_{n\rightarrow +\infty}\sup_{\{h;e(h)\leq\alpha\}}(\sup_{t \in [0;T]\cap [T_{k};T_{k+1}[}
| X^{n}_{h}(t)-X_{h}(t)|)=0,
\end{equation*}
This implies that for any $T>0$
\begin{equation*}
\lim_{n\rightarrow +\infty}\sup_{\{h;e(h)\leq\alpha\}}(\sup_{0\leq t\leq T} | X^{n}_{h}(t)-X_{h}(t)|)=0.
\end{equation*}
 Lemma \ref{expo2} is proved.
\end{proof}

\subsection{The large deviations}
The following theorem is the main result of this section. It ensures that the unique strong solution of
 SDE \eqref{1} satisfies a large deviations principle of Freidlin-Wentzell's type.

\begin{theorem} \label{ThmGD}

Let $\sigma$ and b two bounded continuous functions on $\mathbb{R}^{d}$, taking values respectively in
$\mathbb{R}^{d}\times\mathbb{R}^{m}$ and $\mathbb{R}^{d}$, which satisfy the assumption \eqref{H1}.

Let $\varepsilon >0$ and consider the SDE
\begin{equation}\label{7.3}
X^{\varepsilon}(t)=x+\int_{0}^{t}b(X^{\varepsilon}(s))ds+\sqrt{\varepsilon}\int_{0}^{t}\sigma(X^{\varepsilon}(s))dW(s)
\end{equation}
and denote by $\mu_{\varepsilon}$ the law of $\omega\mapsto X^{\varepsilon}(\cdot, \omega)$ on the space 
$C_{x}([0,T],\mathbb{R}^{d})$.

Then, $\{\mu_{\varepsilon}, \varepsilon>0\}$ satisfies a large deviations principle with the following rate function:
\begin{equation*}
I(u)=\inf\{\frac{1}{2}e(g); X_{g}=u\} \mbox{ for } u\in C_{x}([0,T],\mathbb{R}^{d});
\end{equation*}
 namely,
 \\
(i) for any closed subset $C\subset C_{x}([0,T],\mathbb{R}^{d})$,
\begin{equation*}
\limsup_{\varepsilon\rightarrow 0} \varepsilon \log \mu_{\varepsilon}(C)\leq-\inf_{u\in C}I(u),
\end{equation*}
(ii) for any open subset $O\subset C_{x}([0,T],\mathbb{R}^{d})$,
\begin{equation*}
\liminf_{\varepsilon\rightarrow 0} \varepsilon \log \mu_{\varepsilon}(O)\geq-\inf_{u\in O}I(u).
\end{equation*}

\end{theorem}

\begin{proof}{}
Let $n\geq1$ and define the map $F_{n}:C_{0}([0,T],\mathbb{R}^{m}) \rightarrow C_{x}([0,T],\mathbb{R}^{d})$ by
\begin{equation*}
 \left\{
\begin{array}{cc}
 F_{n}(\omega)(0)=x\\ \\
 F_{n}(\omega)(t)=F_{n}(\omega)(k2^{-n/2})+\sigma(F_{n}(\omega)(k2^{-n/2}))(\omega(t)-\omega(k2^{-n/2})) \\
+ b(F_{n}(\omega)(k2^{-n/2}))(t-k2^{-n/2}).
\end{array}
 \right.
\end{equation*}
Note that $F_{n}$ is a continuous map from $C_{0}([0,T],\mathbb{R}^{m})$ into $C_{x}([0,T],\mathbb{R}^{d})$ and  that
$X^{\varepsilon}_{n}(t)=F_{n}(\sqrt{\varepsilon}\omega)(t)$.

By the continuity of $F_{n}$ and the Schilder large deviations principle for $\{\sqrt{\varepsilon}\omega; \varepsilon>0\}$,
the large deviations principle holds for $X_{n}^{\varepsilon}$.

Therefore, Lemma \ref{expo}, Lemma \ref{expo2} and  Theorem 4.2.23 of   \cite{Dembozeitouni} allows us to complete the proof.
\end{proof}

\section{Application to our motivating example}

In this section, we will study our motivating
and guiding example,
\begin{equation} \label{29}
X_{t}=x+\int_{0}^{t} X_{s} \log{|X_{s}|} ds +\int_{0}^{t} X_{s}\sqrt{|\log{|X_{s}|}|} dW_{s}
\end{equation}
where $(W_{t})_{t \geq 0}$ is an $\mathbb{R}$-valued standard Brownian motion and $x\in\mathbb{R}$.

\subsection{Pathwise unique solution}
\begin{proposition}\label{uniciteforte}
Let $T>0$ be fixed. Then, for any given $x\in\mathbb{R}$, the SDE \eqref{29} admits a unique strong solution
$(X_{t}(x))_{0\leq t\leq T}$. Moreover, for any $x,y\in\mathbb{R}$ such that $x<y$ we have almost surely
$X_{t}(x)<X_{t}(y)$ for any $0\leq t\leq T$. In particular, we have :
\[ X_{t}(x)\geq0 \mbox{ for } x\geq0 \mbox{ and } X_{t}(x)\leq0 \mbox{ for } x\leq0\]
and
\[ X_{t}(0)=0,\, X_{t}(1)=1 \mbox{ and } X_{t}(-1)=-1 \mbox{ almost surely}.\]

\end{proposition}

\begin{proof}{}
We set $b(x)=x\log{|x|}$ and $\sigma(x)=x\sqrt{|\log{|x|}|}$. Since the coefficients $\sigma$ and $b$ are continuous, then according to a well-known result
of Skorohod \cite{Skorohod}
the  SDE \eqref{29} has a weak solution up to a lifetime $\zeta$.
Now, since the coefficients $\sigma$ and $b$ satisfy the following growth conditions
\begin{equation}\label{31}
\left\{
\begin{array}{cc}
  {|\sigma (x)|}^{2} \leq C ({|x|}^{2} \log |x| +1) \\ \\
  |b(x)| \leq C(|x| \log |x| + 1)
\end{array}
 \right.
\end{equation}
for $|x|>K$ with some large constant $K$, a criterion of non-explosion in Fang and Zhang \cite{Fang} yields that
the SDE \eqref{29} does not explode in a finite time ($\zeta\equiv\infty$ a.s.).

To get the pathwise uniqueness it is ennough to prove that $\sigma$ and $b$ satisfy conditions \eqref{H1}. For this,
it is suffices by some computations as in \cite{Bahlali} to see that for any integer $N>e$,
we have
\begin{equation}\label{30}
\left\{
\begin{array}{cc}
 |\sigma (x)-\sigma (y)| \leq C \left( \sqrt{\log{N}} |x-y| + \frac{\log{N}}{N} \right) \\ \\
  |b(x)-b(y)|\leq C \left( \log{N} |x-y| + \frac{\log{N}}{N} \right)
\end{array}
 \right.
\end{equation}
for any $|x|,|y|\leq N$.

Indeed, to verify \eqref{30} for the function $b$, it suffices thanks to triangular inequalities to treat separately
the two cases :
$0\leq|x|, |y|\leq \frac{1}{N}$ and $\frac{1}{N}\leq|x|, |y|\leq N$.
In the first case, since the function $|b|$ is increasing on $[0;1/e]$, then for any integer $N>e$,
$$|b(x)-b(y)|\leq|b(x)|+|b(y)|\leq 2\frac{\log{N}}{N},$$
while in the second case by the finite increments theorem applied to $b$, we have
$$|b(x)-b(y)|\leq (1+\log{N})|x-y|.$$
Hence, for any integer $N>e$, we get for any $x,y\in B(N)=\{z\in\mathbb{R}; |z|\leq N\}$
$$|b(x)-b(y)|\leq 2\log{N}|x-y|+2\frac{\log{N}}{N}$$

In order to verify \eqref{30} for the function $\sigma$, we have to consider separately the following four cases:
$0\leq|x|, |y|\leq \frac{1}{N}$, $\frac{1}{N}\leq|x|, |y|\leq 1-\alpha_{N}$, $1-\alpha_{N}\leq|x|, |y|\leq 1+\beta_{N}$
and $1+\beta_{N}\leq|x|, |y|\leq N$, where $\alpha_{N}$ and $\beta_{N}$ are small positive reals such that
$\sigma(1/N)=\sigma(1-\alpha_{N})=\sigma(1+\beta_{N})$ and $|\sigma'(1-\alpha_{N})|\approx|\sigma'(1/N)|.$ For the first
and third cases, we have for any integer $N>e$
$$|\sigma(x)-\sigma(y)|\leq |\sigma(x)|+|\sigma(y)|\leq 2\frac{\log{N}}{N}$$
since the function $\sigma$ is increasing on $[0;1/\sqrt{e}]$, decreasing on $[1/\sqrt{e};1]$
and increasing on $[1;+\infty[$. For the second and the fourth cases, it follows thanks to the finite increments theorem that
for any integer $N>e$ we have
$$|\sigma(x)-\sigma(y)|\leq |\sigma'(1/N)||x-y|\leq c\sqrt{\log{N}}|x-y| \ \ \mbox{ for some positive constant } c.$$
Hence, for any $x,y\in B(N)$
$$|\sigma(x)-\sigma(y)|\leq C\sqrt{\log{N}}|x-y|+C\frac{\log{N}}{N}$$

Now, according to Theorem \ref{thm 2.1} and thanks to the theorem of Yamada and Watanabe \cite{yawa}, a unique strong
solution holds for the SDE \eqref{29}.

The other assertions are direct consequences of Theorem \ref{thm 2.1}, Theorem \ref{thmcompa} and Theorem \ref{non-contact}. 

The Proposition is proved.
\end{proof}

\subsection{Dependence on the initial value}

In this subsection, we mainly prove that the unique strong solution of SDE \eqref{29} produces a stochastic flow of
homeomorphisms from $\mathbb{R}$ into itself.

\begin{proposition}\label{continuiteucp}

Let $x\in\mathbb{R}$ and consider a sequence $(x_{l})_{l\geq0}$ of real numbers which converges to $x$.
Denote by $X_{t}(x_{l})$ and $X_{t}(x)$ the unique
solutions of SDE \eqref{29} starting respectively from $x_{l}$ and $x$.
Then, for any $\varepsilon >0$ fixed, we have
\begin{equation*}
\lim_{l \rightarrow +\infty} \mathbb{P}(\sup_{t \leq T} {|X_{t}(x_{l})-X_{t}(x)|}>\varepsilon)=0.
\end{equation*}
\end{proposition}

\begin{proof}{}
We set $b(x):=x\log{|x|}$ and $\sigma(x):=x\sqrt{|\log{|x|}|}$. For  $R>1$, we set
$\zeta^{l}_{R}:=\inf\{t>0; |X_{t}(x)|>R \mbox{ or } |X_{t}(x_{l})|>R \}$.
We consider a smooth function with compact support
$\varphi_{R}:\mathbb{R}\rightarrow\mathbb{R}$ satisfying
$$0\leq\varphi_{R}\leq1, \ \ \ \ \varphi_{R}(x)=1 \  \mbox{ for } 
|x|\leq R \ \ \  \mbox{and} \ \ \ \varphi_{R}(x)=0 \ \mbox{ for } |x|>R+1.$$
Put \ $\sigma_{R}(x):=\varphi_{R}(x)\sigma(x)$ and $b_{R}(x):=\varphi_{R}(x)b(x)$.  
Let $X_{t}^{R}(x)$ be the solution of the  SDE
\begin{equation*}
 X^{R}_{t}=x+\int_{0}^{t} \sigma_{R}(X^{R}_{s})dW_{s}+\int_{0}^{t}b_{R}(X^{R}_{s})ds.
\end{equation*}
Notice that $\sigma_{R}$ and $b_{R}$ are bounded and satisfy the conditions \eqref{31} and \eqref{30}.
Hence, by pathwise uniqueness we have $X_{t\wedge\zeta^{l}_{R}}(x)=X^{R}_{t}(x)$ \ $a.s.$ for any $0\leq t\leq T$.\\
Since
\begin{equation*}
\begin{split}
 \mathbb{P}(\sup_{t \leq T} {|X_{t}(x_{l})-X_{t}(x)|}>\varepsilon) &\leq
 \mathbb{P}(\sup_{t \leq T} {|X_{t}(x_{l})-X_{t}(x)|}>\varepsilon;\zeta^{l}_{R}\leq T)\\
  &+\mathbb{P}(\sup_{t \leq T} {|X_{t}(x_{l})-X_{t}(x)|}>\varepsilon;\zeta^{l}_{R}>T).
\end{split}
\end{equation*}
Then it is not difficult to see that
\begin{equation*}
 \mathbb{P}(\sup_{t \leq T} {|X_{t}(x_{l})-X_{t}(x)|}>\varepsilon) \leq
 \mathbb{P}(\zeta^{l}_{R}\leq T)+\mathbb{P}(\sup_{t \leq T\wedge\zeta^{l}_{R}} {|X_{t}(x_{l})-X_{t}(x)|}>\varepsilon).
\end{equation*}
Thanks to Markov inequality, it follows that
\begin{equation*}
 \mathbb{P}(\sup_{t \leq T} {|X_{t}(x_{l})-X_{t}(x)|}>\varepsilon) \leq
 \mathbb{P}(\zeta^{l}_{R}\leq T)+\frac{1}{\varepsilon^{2}}
 \mathbb{E}\sup_{t \leq T} |X_{t\wedge\zeta^{l}_{R}}(x_{l})-X_{t\wedge\zeta^{l}_{R}}(x)|^{2}.
\end{equation*}
That is
\begin{equation}\label{arms}
 \mathbb{P}(\sup_{t \leq T} {|X_{t}(x_{l})-X_{t}(x)|}>\varepsilon) \leq
 \mathbb{P}(\zeta^{l}_{R}\leq T)+\frac{1}{\varepsilon^{2}}\mathbb{E}\sup_{t \leq T} |X^{R}_{t}(x_{l})-X^{R}_{t}(x)|^{2}.
\end{equation}
In addition, since $\sigma$ and $b$ satisfy assumption \eqref{31} and thanks to Remark 7.5
in \cite{Zhang}, it follows that
\begin{equation*}
 \mathbb{P}(\sup_{t\leq T}|X_{t}(x)|\geq R)\leq C_{T} e^{\psi(x^{2})}(\log R)^{-1/2}
\end{equation*}
and
\begin{equation*}
 \mathbb{P}(\sup_{t\leq T}|X_{t}(x_{l})|\geq R)\leq C_{T} e^{\psi(x_{l}^{2})}(\log R)^{-1/2}
\end{equation*}
where $\psi$ is a continuous function defined on $\R_+$ by $\psi(v):=\int_{0}^{v} \frac{ds}{1+u\log u}$.\\
Tending $l$ to $+\infty$ in \eqref{arms} then using Theorem \ref{thm 3.2} and the continuity of the function $\psi$,
it follows that
\begin{equation*}
 \lim_{l\rightarrow +\infty}\mathbb{P}(\sup_{t \leq T} {|X_{t}(x_{l})-X_{t}(x)|}>\varepsilon)
 \leq 2C_{T} e^{\psi(x^{2})}(\log R)^{-1/2}.
\end{equation*}
Letting $R$ tends to $+\infty$ in the above inequality, we get
\begin{equation*}
 \lim_{l\rightarrow +\infty}\mathbb{P}(\sup_{t \leq T} {|X_{t}(x_{l})-X_{t}(x)|}>\varepsilon)=0.
\end{equation*}
\end{proof}

\begin{remark}
Since the coefficients of SDE (\ref{29}) satisfy conditions \eqref{31}, then using
Theorem 4.1 of \cite{Zhang} we get
$$\lim_{|x|\rightarrow+\infty}|X_{t}(x)|=+\infty \mbox{ in probability. } $$
\end{remark}

Now, we give the main result of this subsection.

\begin{proposition} \label{flot}
 The  solution of  SDE (\ref{29}) produces a stochastic flow of homeomorphisms on $\mathbb{R}$.

\end{proposition}

\begin{proof}{}
According to \ref{bicontinuite}, the unique strong solution $X_{t}(x)$ of the SDE (\ref{29}) admits a
version which is bi-continuous in the two variables $(t,x)$ a.s.
Thanks to Theorem \ref{non-contact}, the map $x\mapsto X_{t}(x, \omega)$ is strictly
increasing on $\mathbb{R}$ for almost all $\omega$.
It follows that
$\lim_{|x|\rightarrow+\infty}|X_{t}(x)|=+\infty$ for almost all $\omega$.
Indeed, otherwise the map  $x\mapsto X_{t}(x, \omega)$ would be bounded on $\mathbb{R}$ and this contradicts the fact that
$\lim_{|x|\rightarrow+\infty}|X_{t}(x)|=+\infty$ in probability.

Finally, arguing as in Yamada \& Ogura \cite{yaogura}, it comes that the map $x\mapsto X_{t}(x, \omega)$ is a continuous,
one-to-one and onto. This completes the proof.
\end{proof}

\subsection{Large deviations}

We shall prove that the solution of SDE \eqref{29}
satisfies a large deviations principle of Freidlin-Wentzell's type.

\begin{proposition}

For any $\varepsilon >0$, we consider the following one-dimensional SDE
 \begin{equation}\label{moneds}
  X^{\varepsilon}_{t}=x+\int_{0}^{t}X^{\varepsilon}_{s}\log|X^{\varepsilon}_{s}|ds +
  \sqrt{\varepsilon}\int_{0}^{t} X^{\varepsilon}_{s}\sqrt{|\log|X^{\varepsilon}_{s}||}dW(s)
 \end{equation}
where $(W_{t})_{t\geq0}$ is an $\mathbb{R}$-valued Brownian motion and $x\in\mathbb{R}$.
Denote by $\mu_{\varepsilon}$ the law of $\omega\mapsto X^{\varepsilon}(\cdot, \omega)$ on the space $C_{x}([0,T],\mathbb{R})$.
Then, $\{\mu_{\varepsilon}, \varepsilon>0\}$ satisfies a large deviations principle with the following good rate function:
\begin{equation*}
I(f)=\inf\{\frac{1}{2}e(g); X_{g}=f\} \mbox{ for } f\in C_{x}([0,T],\mathbb{R}); \mbox{ namely }
\end{equation*}
(i) for any closed subset $C\subset C_{x}([0,T],\mathbb{R})$,
\begin{equation*}
\limsup_{\varepsilon\rightarrow 0} \varepsilon \log \mu_{\varepsilon}(C)\leq-\inf_{f\in C}I(f),
\end{equation*}
(ii) for any open subset $O\subset C_{x}([0,T],\mathbb{R})$,
\begin{equation*}
\liminf_{\varepsilon\rightarrow 0} \varepsilon \log \mu_{\varepsilon}(O)\geq-\inf_{f\in O}I(f).
\end{equation*}

\end{proposition}

\begin{proof}{}
We proceed as in \cite{Zhang} for unbounded coefficients.

First, we set $\sigma (x):=x\sqrt{|\log|x||}$ and $b(x):=x\log|x|$. 
Since $\sigma$ and $b$  satisfy the growth conditions \eqref{31},
it follows by using Proposition 7.4 of \cite{Zhang} that
\begin{equation}\label{good}
 \lim_{R\rightarrow +\infty} \limsup_{\varepsilon\rightarrow 0} \varepsilon
 \log \mathbb{P}(\sup_{0\leq t\leq T} |X^{\varepsilon}_{t}|\geq R)=-\infty.
\end{equation}
For any $R>0$, we put $m_{R}:=\sup \{|b(x)|, |\sigma (x)|; |x|\leq R \}$, $b_{R}(x):=(-m_{R}-1)\vee b(x) \wedge (m_{R}+1)$ and
$\sigma_{R}(x):=(-m_{R}-1)\vee \sigma(x) \wedge (m_{R}+1)$. Then, for $|x|\leq R$, $b_{R}(x)=b(x)$ and $\sigma_{R}(x)=\sigma(x)$.
Moreover $b_{R}$ and $\sigma_{R}$ satisfy assumption \eqref{H1} and the growth conditions \eqref{31}.
\\
Let $X^{\varepsilon}_{R}(\cdot)$ be the solution to the following SDE
\begin{equation}\label{grdedeviation}
  X^{\varepsilon}_{R}(t)=x+\int_{0}^{t}b_{R}(X^{\varepsilon}_{R}(s))ds +
  \int_{0}^{t} \sigma_{R}(X^{\varepsilon}_{R}(s))dW(s).
\end{equation}
For a function $h$ with $e(h)<+\infty$,  let $X^{h}_{R}(\cdot)$ be the solution to the following ODE
\begin{equation}\label{monedo}
 X^{h}_{R}(t)=x+\int_{0}^{t}b_{R}(X^{h}_{R}(s))ds +
  \int_{0}^{t} \sigma_{R}(X^{h}_{R}(s))\dot{h}(s)ds.
\end{equation}
For $f\in C_{x}([0,T], \mathbb{R})$, we define
\begin{equation*}
 I_{R}(f)=\inf\{\frac{1}{2}e(g); X^{h}_{R}=f\} \mbox{ and } I(f)=\inf\{\frac{1}{2}e(g); X_{h}=f\}
\end{equation*}
where $X_{h}$ is the solution of the following ODE
\begin{equation} \label{ODE12}
dX_{h}(t)=\left( \sigma(X_{h}(t))\dot{h}(t)+b(X_{h}(t)) \right) dt,\, X_{h}(0)=x\in\mathbb{R}.
\end{equation}
If $\sup_{0\leq t\leq T}|X_{h}(t)|\leq R$, then $X_{h}$ solves the ODE \eqref{monedo} up to time $T$.
By the uniqueness of solutions, we see that $X_{h}(t)=X^{h}_{R}(t)$ for each $0\leq t\leq T$.
Therefore for $f\in C_{x}([0,T], \mathbb{R})$ satisfying $\sup_{0\leq t\leq T} |f(t)|\leq R$, we get $I(f)=I_{R}(f)$.

Furthermore, since $\sigma$ and $b$ satisfy the growth conditions \eqref{31}, then using  Lemma 7.6
of \cite{Zhang}, we have
\begin{equation*}
\mbox{ for any } \alpha >0, \ \ \ \  \sup_{\{h; e(h)\leq\alpha\}} \sup_{0\leq t\leq T} |X_{h}(t)| <+\infty.
\end{equation*}

Besides, we recall that the rate function $I$ is a good rate function, i.e. for any $\beta>0$,
the level $Q_{\beta}=\{f; I(f)\leq \beta\}$ is compact.

Let $\mu^{R}_{\varepsilon}$ be the law of $X^{\varepsilon}_{R}(\cdot)$ on $C_{x}([0,T],\mathbb{R})$.
Then, thanks to Theorem \ref{ThmGD},
$\{\mu^{R}_{\varepsilon}, \varepsilon>0\}$ satisfies a large deviations principle with the rate function $I_{R}(\cdot)$.

For $R>0$ and a closet subset $C\subset C_{x}([0,T], \mathbb{R})$, we set $$C_{R}:=C\cap\{f; \sup_{0\leq t\leq T} |f(t)|\leq R\}.$$
Then, $$\mu_{\varepsilon}(C)\leq \mu_{\varepsilon}(C_{R})+\mathbb{P}(\sup_{0\leq t\leq T}|X^{\varepsilon}_{t}|>R).$$
Since $\mu^{R}_{\varepsilon}$ and $\mu_{\varepsilon}$ coincide on the ball $\{f; \sup_{0\leq t\leq T} |f(t)|\leq R\}$, 
it follows that
$$\mu_{\varepsilon}(C)\leq \mu_{\varepsilon}^{R}(C_{R})+\mathbb{P}(\sup_{0\leq t\leq T}|X^{\varepsilon}_{t}|>R).$$
By large deviations principle for $\{\mu_{\varepsilon}^{R}, \varepsilon>0\}$, we have
$$\limsup_{\varepsilon\rightarrow 0} \varepsilon \log \mu_{\varepsilon}^{R}(C_{R})
\leq -\inf_{f\in C}\{I_{R}(f)\}\leq -\inf_{f\in C} \{I(f)\}.$$
Hence
$$\limsup_{\varepsilon\rightarrow 0} \varepsilon \log \mu_{\varepsilon}(C)\leq (-\inf_{f\in C}\{I(f)\})\vee
(\limsup_{\varepsilon\rightarrow 0} \varepsilon \log \mathbb{P}(\sup_{0\leq t\leq T} |X^{\varepsilon}_{t}|\geq R)).
$$
Using \eqref{good} and letting $R\rightarrow +\infty$, we obtain
$$\limsup_{\varepsilon\rightarrow 0} \varepsilon \log \mu_{\varepsilon}(C)\leq -\inf_{f\in C} I(f),$$
which is the upper bound.

Let G be an open subset of $C_{x}([0,T], \mathbb{R})$. Fix  $\phi_{0}\in G$ and choose $\delta >0$ such that
$$B(\phi_{0},\delta)=\{f;\sup_{0\leq t\leq T}|f(t)-\phi_{0}(t)|\leq \delta \}\subseteq G.$$
Let $R=\sup_{0\leq t\leq T} |\phi_{0}(t)|+\delta$. Since
$$B(\phi_{0},\delta)\subseteq \{f; \sup_{0\leq t\leq T} |f(t)|\leq R\},$$
then
$$-I(\phi_{0})=-I_{R}(\phi_{0})\leq \liminf_{\varepsilon\rightarrow 0} \varepsilon \log \mu^{R}_{\varepsilon}(B(\phi_{0},\delta)),$$
 that is
$$-I(\phi_{0})\leq \liminf_{\varepsilon\rightarrow 0} \varepsilon \log \mu_{\varepsilon}(B(\phi_{0},\delta)).$$
Hence,
$$-I(\phi_{0})\leq \liminf_{\varepsilon\rightarrow 0} \varepsilon \log \mu_{\varepsilon}(G).$$
Since $\phi_{0}$ is arbitrary, it follows that
$$-\inf_{f\in G} I(f) \leq \liminf_{\varepsilon\rightarrow 0} \varepsilon \log \mu_{\varepsilon}(G),$$
which is the lower bound.
The proof is finished.
\end{proof}

\subsection{Other examples}

As a by-product of our guiding example, we give below other examples of SDEs which satisfy our pathwise conditions. 
We also  prove that our conditions for the pathwise uniqueness improve those of \cite{Fang, liang}.

\begin{proposition}

Let $0\leq\beta\leq\frac{1}{2}\leq\alpha\leq1$. Then, the following one-dimensional SDE
\begin{equation} \label{2923}
X_{t}=x+\int_{0}^{t} |X_{s}|^{\alpha}|\log{|X_{s}|}|^{2\beta}ds
+\int_{0}^{t} |X_{s}|^{\alpha}|\log{|X_{s}|}|^{\beta} dW_{s}
\end{equation}
where $(W_{t})_{t \geq 0}$ is an $\mathbb{R}$-valued standard Brownian motion and $x\in\mathbb{R}$, possesses a pathwise
unique solution which has produces a stochastic flow of homeomorphisms on $\mathbb{R}$
and satisfies a large deviation principle of Freidlin-Wentzell type.

\end{proposition}

\begin{proof}{}
It suffices to prove as in the proof of Proposition \ref{uniciteforte} that the coefficients of the SDE \eqref{2923} satisfy
the following
\begin{equation}\label{3012}
\left\{
\begin{array}{cc}
 |\sigma (x)-\sigma (y)| \leq C \left(\sqrt{\log{N}} |x-y| + \frac{\log{N}}{N^{\alpha}} \right) \\ \\
  |b(x)-b(y)|\leq C \left( \sqrt{\log{N}}|x-y| + \frac{\log{N}}{N^{\alpha}} \right)
\end{array}
 \right.
\end{equation}
for any $|x|,|y|\leq N$ and the following growth conditions
 \begin{equation}\label{31124}
\left\{
\begin{array}{cc}
  {|\sigma (x)|}^{2} \leq C ({|x|}^{2} \log |x| +1) \\ \\
  |b(x)| \leq C(|x| \log |x| + 1)
\end{array}
 \right.
\end{equation}
for $|x|>K$ and some large constant $K$.

\end{proof}

We prove now that our conditions for the pathwise uniqueness improve those of
Fang and Zhang \cite{Fang} and also Liang \cite{liang2}.

\begin{proposition}

Let $\sigma:\mathbb{R}^{d}\rightarrow\mathbb{R}^{d}\times\mathbb{R}^{m}$ and $b:\mathbb{R}^{d}\rightarrow\mathbb{R}^{d}$ be
respectively matrix valued and vector-valued continuous functions such that
\begin{equation}\label{nonFangZhanggeneralise}
\left\{
\begin{array}{cc}
 ||\sigma(x)-\sigma(y)|| \leq C |x-y|\sqrt{\log\frac{1}{|x-y|}} \\ \\
 |b(x)-b(y)|\leq C |x-y|\log\frac{1}{|x-y|}
\end{array}
 \right.
\end{equation}
for any $|x-y|<1$.

Then, $\sigma$ and $b$ satisfy the hypothesis $\eqref{H1}$, i.e. our conditions for the pathwise uniqueness improve
those of Fang and Zhang \cite{Fang}.

\end{proposition}

\begin{proof}

Let $x, y\in\mathbb{B}(N)=\{ z\in\mathbb{R}^{d};|z|\leq N \}$ with $|x-y|<1$ for any integer $N>e$.
Then, since $0\leq |x-y|<1<N$, it follows thanks to \eqref{nonFangZhanggeneralise} that
\begin{equation}\label{nonFangZhanggeneraliseproof}
\left\{
\begin{array}{cc}
 ||\sigma(x)-\sigma(y)|| \leq C |f(|x-y|)-f(0)| \\ \\
 |b(x)-b(y)|\leq C |g(|x-y|)-g(0)|
\end{array}
 \right.
\end{equation}
where for $0\leq u\leq 1$, $f(u)=u\sqrt{-\log u}$ and $g(u)=-u\log{u}$. Now, thanks to \eqref{30} the following holds
\begin{equation}\label{30finall}
\left\{
\begin{array}{cc}
 ||\sigma (x)-\sigma (y)|| \leq C \left( \sqrt{\log{N}} |x-y| + \frac{\log{N}}{N} \right) \\ \\
 |b(x)-b(y)|\leq C \left( \log{N} |x-y| + \frac{\log{N}}{N} \right)
\end{array}
 \right.
\end{equation}
for any $|x|,|y|\leq N$ such that $|x-y|<1$. The proof is finished.
\end{proof}

\begin{remark}
The coefficients $\sigma(x) = x\sqrt{\log{|x|}}$ and $b(x)= x\log{|x|}$ of SDE (\ref{29}) are not covered by the papers \cite {Fang, Zhang}.
\end{remark}

\begin{proof}
We give only the proof for $b$. The proof for $\sigma$ goes similarly.
Assume that $x\log{|x|}$  satisfies the conditions of \cite {Zhang} for instance. Then, there exist $C>0$, $c_0 \in ]0, 1]$ and a positive 
$\mathcal{C}^1$ function $r$ such that for every $x$, $y$ satisfying $|x-y| \leq c_0$ ,
\begin{equation}\label{fz}
|x\log |x| - y\log |y|| \leq C |x-y| r(|x-y|^2).
\end{equation}
We take $c_0 = 1$ for simplicity. Let $x > 1$ be large enough and $y = x+1$. From  inequality (\ref{fz}), we have
\begin{equation*}
|(x+1)\log |x+1| - x\log |x|| \leq C r(1)
\end{equation*}
Hence, according to the finite increments theorem, there exists $\theta \in [x, \ x+1]$ such that

\begin{equation*}
|1 + \theta\log \theta| \leq C r(1)
\end{equation*}
Since $x \leq \theta$ and the function $\log$ is increasing, we deduce that
\begin{equation}\label{fz3}
1 + x\log x \leq C r(1)
\end{equation}
Since $x$ is arbitrary, the previous inequality is not possible.
\end{proof}

\begin{remark}
Arguing  as in the previous proposition, we prove that our conditions for the pathwise uniqueness improves also those
of  \cite{liang2}.
\end{remark}

\vskip 0.2cm\noindent \textbf{Acknowledgement.}
The authors thank the referee for the remarks which have led to the improvement of the paper.

\end{document}